\documentclass[aap,noinfoline]{imsart}
\usepackage{latexsym,epsfig,amssymb,amsmath,amsthm,color,url}
\RequirePackage{natbib}
\usepackage{hyperref}
\usepackage{enumerate}
\setattribute{journal}{name}{}

\allowdisplaybreaks 

\numberwithin{equation}{section}
\newtheorem{thm}{Theorem}[section]
\newtheorem{cor}[thm]{Corollary}
\newtheorem{prop}[thm]{Proposition}
\newtheorem{lemma}[thm]{Lemma}

\theoremstyle{definition}
\newtheorem{defn}{Definition}[section]
\newtheorem{assum}{Assumption}[section]
\newtheorem{notn}{Notation}[section]

\theoremstyle{remark}
\newtheorem{remark}{Remark}[section]

\def\beq{\begin{equation}}
\def\eeq{\end{equation}}
\def\bt{\begin{thm}}
\def\et{\end{thm}}
\def\bc{\begin{cor}}
\def\ec{\end{cor}}
\def\bp{\begin{prop}}
\def\ep{\end{prop}}
\def\bl{\begin{lemma}}
\def\el{\end{lemma}}
\def\bd{\begin{defn}}
\def\ed{\end{defn}}
\def\bass{\begin{assum}}
\def\eass{\end{assum}}
\def\bark{\begin{remark}}
\def\eark{\end{remark}}
\def\bn{\begin{notn}}
\def\en{\end{notn}}
\def\ba{\begin{enumerate}[(a)]}
\def\bei{\begin{enumerate}[(i)]}
\def\be{\begin{enumerate}[(1)]}
\def\ee{\end{enumerate}}
\def\bi{\begin{itemize}}
\def\ei{\end{itemize}}
\def\Bar{\overline}
\def\RR{\ensuremath{\mathbb{R}}}

\def\ii{\item}

\def\Lp{\left(}
\def\Rp{\right)}
\def\LP{\left\{ }
\def\RP{\right\}}
\def\LT{\left[}
\def\RT{\right]}
\def\mid{\middle\vert}

\def\epf{\end{proof}}

\DeclareMathOperator{\Exp}{\mathbb{E}}
\DeclareMathOperator{\Prob}{\mathbb{P}}
\def\b{\boldsymbol}

\def\C{\b{C}} 
\def\R{\b{R}} 
\def\Chi{\b{\chi}} 
\def\N{\b{N}} 
\def\FF{\mathcal{F}} 

\def\h{h_{\H}}

\newcommand{\mynorm}[1]{\rho\Lp #1\Rp} 

\newcommand{\Lone}[1]{\left\|#1\right\|_{\ell^1}}

\newcommand{\csd}[1]{\frac{\C_{#1}}{S_{#1}}}

\newcommand{\csl}[1]{\C_{#1}/S_{#1}}

\def\cslnmi{C_{n-1,i}/S_{n-1}}

\def\mgdh{\b{D}}

\def\mgdhc{D}

\def\bxi{\b{\xi}} 

\def\H{\b{H}} 
\def\rpf{\b{\zeta}_{\H}} 
\def\lpf{\b{\pi}_{\H}} 
\def\pfval{\lambda_{\H}} 
\def\U{\b{U}} 
\def\V{\b{V}} 
\def\J{\b{J}} 
\def\I{\b{I}} 

\def\X{\b{X}} 
\def\Y{\b{Y}} 
\def\Z{\b{Z}} 
\def\x{\b{x}} 
\def\U{\b{U}}
\def\bbeta{\b{\beta}} 

\newcommand{\CExp}[2]{\Exp\Lp #1\mid\FF_{#2}\Rp}
\newcommand{\MGD}[2]{#1 - \CExp{#1}{#2}}

\newcommand{\ieq}[1]{\b{I}_{G,#1}}

\def\bb{\b{B}}

\newcommand{\st}[2]{\Bar{t}_{#1}\Lp #2\Rp}
\newcommand{\ST}[2]{\Bar{\tau}_{#1}\Lp #2\Rp}
\newcommand{\stb}[2]{\underline{t}_{#1}\Lp #2\Rp}
\newcommand{\STb}[2]{\underline{\tau}_{#1}\Lp #2\Rp}

\def\pt{t}
\def\0{\b{0}}
\def\1{\b{1}}
\def\h{h_{\H}}
\def\Xb{\Bar{\b{X}}}
\def\e{\b{e}}

\def\tX{\widetilde{\b{X}}}
\def\tG{\widetilde{G}}
\def\te{\widetilde{\b{e}}}

\def\hH{\widehat{h}_{G}}

\def\hX{\widehat{\X}}
\def\htX{\widehat{\widetilde{\X}}}

\def\tbb{\widetilde{\bb}}
\DeclareMathOperator{\osc}{Osc}

\def\themax{\max_{p_n\leq m\leq q_n}}
\def\themin{\min_{p_n\leq m\leq q_n}}

\begin{document}

\begin{frontmatter}
\title{Stochastic Approximation with Random Step Sizes and Urn Models with Random Replacement Matrices Having Finite Mean}
\runtitle{Random Step Size Stochastic Approximation and Urn Model}

\begin{aug}
\author{\fnms{Ujan} \snm{Gangopadhyay}\ead[label=e1]{ujangangopadhyay@gmail.com}}
\and
\author{\fnms{Krishanu} \snm{Maulik}\corref{}\thanksref{t2}\ead[label=e2]{krishanu@isical.ac.in}\ead[label=u1,url]{http://www.isical.ac.in/$\sim$krishanu}}

\thankstext{t2}{The research of the second author was partially supported by an Unrestricted Research Grant from Microsoft Research India. The authors like to thank B.V.\ Rao and Vivek S.\ Borkar for helpful and enlightening discussions.}

\runauthor{U. Gangopadhyay and K. Maulik}

\affiliation{University of Southern California and Indian Statistical Institute}

\address{Ujan Gangopadhyay\\ Department of Mathematics\\ University of Southern California\\ Los Angeles, CA\\ USA\\ \printead{e1}}
\address{Krishanu Maulik\\ Theoretical Statistics and Mathematics Unit\\ Indian Statistical Institute\\ Kolkata\\ India\\ \printead{e2}\\ \printead{u1}}
\end{aug}

\begin{abstract}
Stochastic approximation algorithm is a useful technique which has been exploited successfully in probability theory and statistics for a long time. The step sizes used in stochastic approximation are generally taken to be deterministic and same is true for the drift. However, the specific application of urn models with random replacement matrices motivates us to consider stochastic approximation in a setup where both the step sizes and the drift are random, {but the sequence is uniformly bounded. The problem becomes interesting when the negligibility conditions on the errors hold only in probability.} We first prove {a result} on stochastic approximation in this setup, which {is} new in the literature. Then, as an application, we study urn models with random replacement matrices.

In the urn model, the replacement matrices need neither be independent, nor identically distributed. We assume that the replacement matrices are only independent of the color drawn in the same round conditioned on the entire past. We relax the usual second moment assumption on the replacement matrices in the literature and require only first moment to be finite. We require the conditional expectation of the replacement matrix given the past to be close to an irreducible matrix, in {an} appropriate sense. We do not require any of the matrices to be balanced or nonrandom. {We prove convergence of the proportion vector, the composition vector and the count vector in $L^1$, and hence in probability.} It is to be noted that the related differential equation is of Lotka-Volterra type and can be analyzed directly.
\end{abstract}

\begin{keyword}[class=MSC]
\kwd[Primary ]{62L20}
\kwd[; secondary ]{60F15}
\kwd{60G42}
\end{keyword}

\begin{keyword}
\kwd{urn model}
\kwd{random replacement matrix}
\kwd{balanced replacement matrix}
\kwd{irreducibility}
\kwd{stochastic approximation}
\kwd{random step size}
\kwd{random drift}
\kwd{uniform integrability}
\kwd{Lotka-Volterra differential equation}
\end{keyword}

\end{frontmatter}

\section{Introduction} \label{sec: intro}

In this article, we apply the method of stochastic approximation to prove convergence results in a generalized urn model. We begin with brief reviews of both stochastic approximation and urn model. Stochastic approximation is a useful technique to analyze various stochastic models for more than half a century now. We prove a significant generalization of the technique by allowing the step sizes and the drift to be random. Urn model is also a very well studied stochastic model for nearly a century. We use the extension of stochastic approximation technique, as proved by us in this article, to study the asymptotic behaviour of a generalization of urn models with random replacement matrices.

\begin{subsection}{Stochastic approximation}
Broadly speaking, stochastic approximation is a method of finding root(s) of an unknown function. The function may be the output of some experiment or procedure where we can observe its value at certain points with errors but it is not possible to know the function explicitly. Of course there are plenty of recursive algorithms which address this problem. But stochastic approximation is particularly useful when the value of the function can only be observed with some random error. The method was introduced by \citet{RobbinsMonro}. Essentially their algorithm is as follows. Suppose the function under consideration is a continuous function $h : \RR^K \to \RR^K$ which we call the \textit{drift} following \citet{Borkar}. We want to construct a sequence $\Lp \x_n \Rp_{n=0}^\infty$ iteratively which will converge to a root of $h$. Suppose we can only observe $h \Lp \x_n \Rp + \bbeta_n$ where $\bbeta_n$ is a random error term. Naturally we want to assume that $\bbeta_n$ does not add any significant deviation on an average, given the past $\LP \x_m : m < n \RP$, i.e., $\Lp \bbeta_n \Rp_{n=0}^\infty$ is a martingale difference sequence. Then a sequence of positive real numbers $\Lp a_n \Rp_{n=0}^\infty$ is chosen and the sequence $\Lp \x_n \Rp$ is generated via the recursion
\[
\x_{n+1} = \x_n + a_n \Lp h(\x_n) + \bbeta_n \Rp.
\]
The ingenuity of stochastic approximation is in choosing $\Lp a_n \Rp_{n=0}^\infty $ such that $a_n\to 0$, but $\sum_n a_n = \infty$ and $\sum_n a_n^2 < \infty$. It is intuitively clear that the sequence can not get stuck at a point which is not a root of $h$ since $\sum_n a_n = \infty$. Also since $\sum_n a_n^2 < \infty$, the series $\sum_n a_n\bbeta_n$ is convergent under mild assumptions. In dimension one, i.e., $K=1$, it is easy to prove that the sequence $\Lp \x_n \Rp$ must converge to the zero-set of $h$ if it is non-empty. In dimension more than one, the situation is trickier since, loosely speaking, there are more ways to escape the zero-set. Of course it is indeed not true in general that the sequence converges to the zero-set of $h$.

In higher dimension, the most popular technique for analysing the asymptotic behaviour of the sequence $\Lp \x_n \Rp_{n=0}^\infty$ is the so called ODE method. We shall follow this method inspired by \citet{KushnerClark}. Let $\pt_n := \sum_{m=0}^{n-1} a_m$ for all $n>1$ and $\pt_0:=0$. We see that the evolution equation can be written as $\x_{n+1} - \x_n = ( \pt_{n+1} - \pt_n ) ( h(\x_n) + \bbeta_n )$ which can be considered as a discrete approximation to the ODE $\dot{\x} = h(\x)$. We consider the piecewise linear interpolation $\X$ of the sequence $\Lp \Lp \pt_n, \x_n \Rp \Rp_{n=0}^\infty$ and compare it to the solutions of the above ODE. In particular, if the solution the ODE is unique and is a constant function, then $\x_n$ will converge to the constant, which will also be the zero of the function $h$.

Instead of considering the ODE, the set $A = \cap_t \Bar{ \LP \X(s) : s\geq t \RP }$ may be considered and derive various properties of the set $A$ which then can be used to obtain asymptotic properties of $\Lp \Lp\pt_n, \x_n \Rp \Rp_{n=0}^\infty$. This is usually called the dynamical system approach. We refer the reader to \citet{benaim} and \citet{Borkar}.

In much of the literature, the step size $\Lp a_n \Rp$ and the drift $h$ are taken to be deterministic. We allow both of them to be random and prove version of stochastic approximation where the bounded sequence $(\x_n)$ being approximated converges in probability (actually in $L^1$) under suitable set of assumptions.
\end{subsection}

\begin{subsection}{Urn model}
Urn model is one of the simplest and most useful models considered in probability theory. Since its introduction by \cite{PolyaEggenberger}, the model has seen various extensions and generalizations. However, the basic framework remains the same. We start with an urn with balls of different colors. We shall allow finitely many $K$ colors. We are interested in the composition of the urn after different trials, or equivalently, the number of balls of different colors after the trials. We shall denote the composition vector after $n$-th trial as $\C_n=\Lp C_{n1},\dots,C_{nK}\Rp$, {where $\C_{ni}$ is the count of $i$-th color after $n$-th trial.} Speaking of the number of the balls, it seems natural to consider $\C_n$ as a vector of nonnegative integers. However, as we shall see, the assumption of integer values is not very important, and we shall allow each of $C_{ni}$ to take nonnegative real values. (Traditionally, the color is chosen by drawing a ball at random, noting its color and returning the ball, leading to the importance of the ``number" of each color.) At each trial, a color is chosen at random, namely, given the present composition, the conditional probability of choosing a color will be proportional to the share of the color in the composition. The total content after $n$-th trial is {$S_{n}=\sum_i C_{ni}$.} At every trial $n$, we have a $K\times K$ replacement matrix $\R_n$ with nonnegative entries, which are possibly random. If a ball of color $i$ is drawn then for each color $j$, $R_{nij}$ balls of color $j$ are added to the urn. The main question of interest in urn models is the asymptotic behavior of the color proportion $\csl{n}$ and the growth rate of the composition vector $\C_n$. We shall also consider the count vector $\N_n$, \label{def N} which counts the number of times each color is drawn, and study its growth rate.

\cite{PolyaEggenberger} and \cite{Polya} considered a fixed deterministic sequence of replacement matrix of the form $\lambda\I$ for some $\lambda>0$. The corresponding urn model is known as \textit{Polya Urn}. Here, the proportion vector $(\csl{n})$ converges almost surely to a random vector jointly distributed as Dirichlet with parameter $\lambda^{-1}\csl{0}$, cf.\ \citet{Freedman}. Since we add same number of balls at each trial, $S_n$ and $\C_n$ both grows linearly. This has been generalized in \textit{Friedman urn}, where {$K=2$ and} the replacement matrix sequence is again fixed and deterministic of the form $\scriptsize{\begin{pmatrix}\alpha\;&\beta\\
\beta\;&\alpha\end{pmatrix}}$. When $\beta\neq 0$, \cite{Freedman} proved, using martingale techniques, that the proportion vector $\csl{n}$ converge almost surely to $(1/2,1/2)$. To understand the difference in behavior of these two models we need the concept of irreducibility.

Recall that a square matrix $\b{A}$ with nonnegative entries is called \textit{irreducible} if, for any $i,j$ there exists a positive integer $N\equiv N_{ij}$ such that $A^{(N)}_{ij} >0$, where $A^{(N)}_{ij}$ denotes the $(i,j)$-th entry of $\b{A}^N$. By Perron Frobenius theory the eigenvalue of an irreducible matrix having largest real part is simple and positive. This eigenvalue is called the \textit{dominant eigenvalue}. If the matrix is \textit{balanced}, namely, all the row sums are equal then the dominant eigenvalue is the common row sum. Otherwise, the dominant eigenvalue lies strictly between the largest and the smallest row sums. Also corresponding to the dominant eigenvalue, there exists left and right eigenvectors which have all coordinates positive. The unique probability vector that is a left eigenvector of the dominant eigenvalue will be referred to as the \textit{stationary distribution}. For more details, see, for example, Chapter~1.3 of \cite{Seneta}.

When each $\R_n$ is a fixed deterministic balanced irreducible replacement matrix, \cite{Gouet} proved, again using martingale techniques, that the proportion vector $(\csl{n})$ converges almost surely to the stationary distribution of the replacement matrix. The case of fixed deterministic balanced, but not necessarily irreducible, replacement matrix was considered in \cite{DasguptaMaulik}, where growth rate of each coordinate of the composition vector was determined. \cite{AthreyaKarlin} considered the case where $(\R_n)$'s are i.i.d.\ with integer coefficients and $\Exp(\R_n)$ is irreducible. Then, by embedding into a continuous time multitype branching process, under the condition that $\Exp(R_{nij}\log R_{nij})<\infty$ for all $i,j$, they showed that $(\csl{n})$ converges almost surely to the stationary distribution of $\Exp(\R_1)$. See also \cite{AthreyaNey}.

However, in many applications it is too restrictive to assume deterministic or i.i.d.\ sequence of replacement matrices. One such application is that of adaptive designs used in clinical trials, as considered by, e.g., \cite{BaiHu}, \cite{LaruellePages} or \cite{Zhang}. They allowed the replacement matrix to depend on the choice of the color drawn in the entire past. They also assumed the conditional expectation of the replacement matrix given the entire past to be close to a matrix $\H$ in some appropriate norm. In literature, $\H$ is assumed to be deterministic, but using stochastic approximation with random step sizes and drift, we can allow $\H$ to be random. Such models have been analyzed in literature depending on moment conditions on the replacement matrix. For example, assuming that, for some $\delta>0$, the entries of $\R_n$ have conditional $(2+\delta)$ moments bounded, and $\H$ is deterministic and balanced, \cite{BaiHu} proved, using martingale techniques, $\C_n$ grows linearly. \cite{LaruellePages} still considered deterministic and balanced limit $\H$, but required only second conditional moment of $\R_n$ to be bounded. Under these hypotheses, they again proved linear growth of $\C_n$, using stochastic approximation algorithm. \cite{Zhang} removed the balanced assumption and considered the case where $\H$ can be unbalanced but nonrandom, again using stochastic approximation techniques and obtained the same convergence result under same bounded conditional second moment assumption as in \cite{LaruellePages}. All these articles further investigated Central Limit Theorem behavior and it was thus natural to assume the existence of the second moments. However, to study Law of Large Number type growth behavior of $\C_n$, it should be enough to assume the existence of the first moment alone. It may be noted here that, while \cite{AthreyaKarlin} obtained almost sure convergence in the case of i.i.d.\ integer valued replacement matrices under so called $L \log L$ condition, \cite{Zhang12} used stochastic approximation methods to prove almost sure convergence of linearly scaled configuration vector in the setup of \cite{BaiHu} only under $L (\log L)^p$ condition, for some $p>1$. \cite{Zhang12} also proved the result under $L \log L$ condition, but only for i.i.d.\ (not necessarily integer valued) replacement matrices.

We analyze this problem using stochastic approximation with random step size sequence $(1/S_{n+1})_n$ unlike the deterministic one $(1/(n+1))_n$, used in \cite{LaruellePages} or \cite{Zhang}. While this requires more careful and subtle analysis of the step size, the associated differential equation becomes easier to handle. The differential equation is a first order quadratic one, more precisely, of Lotka-Volterra type, and can be studied more explicitly than in \cite{Zhang}. The differential equations in our analysis, as well as that of \citeauthor{Zhang}, reduce to that of \citeauthor{LaruellePages}, when $\H$ is balanced. Due to our choice of step size sequences, we primarily study the proportion vector $\C_n/S_n$ instead of the linearly scaled composition vector $\C_n$ in \cite{Zhang}. We derive the linear scaling of the composition vector and the count vector from the behavior of the proportion vector. We assume existence of only the first moments for the replacement matrices and show $L^1$ and hence in probability convergence.
\end{subsection}

In Section~\ref{sec: sa}, we state and prove the result on stochastic approximation with random step size as well as drift. Theorem~\ref{thm: SA weak}, {is} among the main contributions of this article. In Section~\ref{sec: lv}, we rewrite the evolution equation of urn model in a form appropriate for stochastic approximation. We then study the related Lotka-Volterra differential equation in Proposition~\ref{Lotka} and show uniqueness of its probability solution. The analysis of this ODE is strikingly simpler and straight forward in comparison to the ODE considered by \cite{Zhang}. We also collect some useful results and prove required conditions on the step size sequence. Finally, in Section~\ref{sec: main}, we analyze the error terms and the main convergence results are stated and proved. All the vectors considered are row vectors, while for a row vector $\x$, its transpose will be denoted by $\x^T$. The vector with all components $1$ will be denoted by $\1$, while the dimension will be clear from the context. Also, for an event $A$, we shall denote its indicator function by $\chi(A)$.

In summary, this article makes three important contributions, namely, an extension to stochastic approximation algorithm for bounded sequence, when step size as well as drift are random; rewriting the stochastic approximation for an urn model with unbalanced random replacement matrix in a fashion so that the corresponding differential equation is of Lotka-Volterra type and can be analyzed directly through a simple change of variable; and finally a complete analysis of urn models with balls of finitely many colors and random replacement matrix, with only finite first moment condition.

\section{Stochastic Approximation with Random Step Size and Drift} \label{sec: sa}

In this Section we state and prove one of the main results of this article regarding stochastic approximation with random step size $(a_n)_{n=0}^\infty$, satisfying the usual conditions, namely, $a_n> 0$, $a_n\to 0$ and $\sum_n a_n = \infty$ almost surely, and a random drift $h$, which will depend on a Polish space $\mathfrak G$ valued argument $G$ as well. We shall denote the function as $(G, \x) \in \mathfrak G \times \RR^K \mapsto h_G(\x) \in \RR^K$, which is jointly continuous in $(G, \x)$. Let $(\X_n)_{n=0}^\infty$ be a bounded sequence of random variables in $\RR^K$, which lie in a (fixed) bounded and (without loss of generality) closed and convex subset $S$ of $\RR^K$. In fact, it is enough to define the drift $h$ on $G\times S$ only. The error sequence $(\bbeta_n)_{n=0}^\infty$  are also $\RR^K$-valued random variable, which, together with a $\mathfrak G$-valued random variable $G$ is defined on the same probability space as the sequences $(a_n)$ and $(\X_n)$. The sequence $(\X_n)$ will satisfy the stochastic approximation equation
$$\X_{n+1} = \X_n + a_n h_G (\X_n) + a_n \bbeta_n.$$

\bark \label{rem: mble}
Observe that $\sup_{\x\in S} h_G(\x)$ is measurable. Further, for each $G$, $h_G$ is continuous in $\x$. Hence, $S$ being closed and bounded, we have $\sup_{\x\in S} h_G(\x)$ finite for each fixed $G$ and hence it is a finite random variable.
\eark

The model and the assumptions (including those involving the error sequence, cf.~Subsection~\ref{subsec: assmp}) are motivated by \citet{KushnerClark}. \citet{KushnerClark} also provides an analogous result for convergence in probability, when the step sizes and the drift are deterministic. However, we substantially improve the result in Theorem~\ref{thm: SA weak} by allowing the step sizes and the drift to be random. This is the main contribution of this section, as well as one of the main contributions of this article. It is also significantly used later in the application to urn models.

\subsection{Linear interpolates}
To study the stochastic approximation we consider the piecewise linear interpolations of the sequence $(\X_n)$, taking values in $\RR^K$. Let $\pt_0:=0$ and $\pt_n := \sum_{i=0}^{n-1} a_i$ be the partial sum sequence. The continuous piecewise linear interpolation $\X^0$ of $((\pt_n, \X_n))_{n=0}^\infty$ is given by
$$\X^0(t) := \begin{cases}
\X_n\frac{t_{n+1}-t}{a_n}+\X_{n+1}\frac{t-\pt_n}{a_n} &\text{if } t\in\LT\pt_n,\pt_{n+1}\RT\\
\X_0 &\text{ if } t\leq t_0=0.
\end{cases}$$
Similarly define the piecewise constant interpolation $\Xb^0$ of $((\pt_n, \X_n))_{n=0}^\infty$, as
$$\Xb^0(t) := \begin{cases}
\X_n &\text{ if } t\in\LT\pt_n,\pt_{n+1}\Rp\\
\0 &\text{ if } t<t_0=0.
\end{cases}$$
Next, $\X^n$ and $\Xb^n$ are the corresponding translates given by $\X^n(t) := \X^0 (t_n+t)$ and $\Xb^n(t) := \Xb^0 (t_n+t)$. Since $S$ is convex, observe that $\X^0(t)$ lies in $S$ for all $t$, and so happens for $\X^n(t)$ too for all $n$ and $t$.

The partial sums of the error terms in the stochastic approximation equation will be denoted by $\bb_0:=\0$ and for $n>0$, $\bb_n:=\sum_{m=0}^{n-1}a_m\bbeta_m$. The continuous piecewise linear interpolation $\bb^0$ of $((\pt_n, \bb_n))_{n=0}^\infty$ is given by
$$\bb^0(t) := \begin{cases}
\bb_n\frac{t_{n+1}-t}{a_n}+\bb_{n+1}\frac{t-\pt_n}{a_n} &\text{ if } t\in\LT\pt_n,\pt_{n+1}\RT\\
\bb_0=\0 &\text{ if } t\leq t_0=0.
\end{cases}$$
Its translates will be $\bb^n(t):=\bb^0\Lp t+\pt_n\Rp-\bb^0\Lp\pt_n\Rp=\bb^0\Lp t+\pt_n\Rp-\bb_n$.

Finally define
$$\e^n\Lp t\Rp:=\int_0^t h_G\Lp\Xb^n(s)\Rp ds - \int_0^t h_G\Lp\X^n\Lp s\Rp\Rp ds.$$

\subsection{Level crossing times}
The linear interpolations above can be written in terms of level crossing times. For $n>0$ and $t>0$, define the forward and backward level crossing times as:
\begin{align}
  \ST{n}{t} &:=\max\LP m\ge n:a_n+\cdots+a_{m-1}\le t\RP \label{def: tau up}\\
  \text{and } \STb{n}{t} &:=\min\LP 0\leq m\le n-1: a_{m+1}+\cdots+a_{n-1}\leq t\RP, \label{eq: tau down}
\end{align}
where empty sums are taken to be zero. The following properties of $\ST{n}{t}$ and $\STb{n}{t}$ are easy to establish:
\begin{lemma} \label{lem: prop tau}
The following statements about $\ST{n}{t}$ and $\STb{n}{t}$ hold on a set of probability $1$:
\bei
\ii For all $n>0$, $t>0$, $\ST{n}{t}$ takes values in $\LP n,n+1,\dots\RP$. On the event $\LT a_n \le t \RT$, $\ST{n}{t}> n$.
\ii For all $n>0$, $t>0$, $\STb{n}{t}$ takes values in $\LP 0,1,\dots,n-1\RP$. On the event $\LT a_{n-1} \le t \RT$, $\STb{n}{t}<n-1$.
\ii On the event $\LT a_n \le t \RT$, $a_n + \cdots + a_{\ST{n}{t}-1} \leq t$, but $a_n + \cdots + a_{\ST{n}{t}} > t$. Further on this event,
$$\ST{n}{t} =\max\LP m> n:a_n+\cdots+a_{m-1}\le t\RP>n.$$
\ii On the event $\LT a_{n-1} \le t \RT$, $a_{\STb{n}{t}+1} + \cdots + a_{n-1} \leq t$ and
$$\STb{n}{t} =\min\LP 0\leq m<n-1: a_{m+1}+\cdots+a_{n-1}\leq t\RP < n-1.$$ On the event $\LT t<\pt_n\RT$, $a_{\STb{n}{t}} + \cdots + a_{n-1} > t$.
\ii For all $n>0$, $t>0$, $\ST{n}{t}=\max\LP m\geq n:\pt_m\leq \pt_n+t\RP$. \label{tau up}
\ii On the event $\LT \pt_n>t\RT$, $\STb{n}{t}=\max\LP 0\leq m\leq n-1:\pt_m<\pt_n-t\RP$. On the event $\LT\pt_n\leq t\RT$, $\STb{n}{t}=0$. \label{tau dn}
\ii For all $t>0$, as $n\to\infty$, we have $\ST{n}{t}\uparrow\infty$ and $\STb{n}{t}\uparrow\infty$.
\ii For all $t>0$, eventually for large $n$, we have $\ST{n}{t}<\ST{n}{2t}$ and $\STb{n}{t}>\STb{n}{2t}$. \label{scale}
\ee
\end{lemma}
\begin{proof}
Proofs of the statements (i)-(vi) are easy and follow from the definition of $\ST{n}{t}$ and $\STb{n}{t}$. So is the first half of the statement (vii).

Since $a_n\to 0$, for each $t$, the event $[a_{n-1}\leq t]$ holds eventually and from \emph{(iv)}, we have $\pt_n \leq \pt_{\STb{n}{t}+1} + t$. Hence, if $\STb{n}{t}$ is bounded, so is $\pt_n$, contradicting $\pt_n\to\infty$, which gives the remainder of the statement (vii).

For the statement (viii), observe that, if $\ST{n}{t} = \ST{n}{2t}$, then $a_{\ST{n}{t}}> t$. Hence on the event $[a_{\ST{n}{t}}\le t]$, we have $\ST{n}{t} < \ST{n}{2t}$. The event happens all but finitely often with probability $1$, since $a_n\to 0$ and $\ST{n}{t}\to\infty$. Similarly the other strict inequality holds on the event {$[a_{\STb{n}{t}}\le t]$}.
\end{proof}

Using the level crossing times, we can establish integral equations for the linear interpolations.
\bl \label{lem: int eqn}
The following relations hold for all $n\ge 0$,
\begin{align}
\X^0\Lp t\Rp&=\X^0\Lp 0\Rp + \int_0^t h_G\Lp\Xb^0\Lp s\Rp\Rp ds + \bb^0(t) && \mbox{for } t\geq 0,\nonumber\\
\X^n\Lp t\Rp&=\X^n\Lp 0\Rp+\int_0^t h_G\Lp\Xb^0\Lp\pt_n+s\Rp\Rp ds + \bb^n(t) && \mbox{for } t\geq -\pt_n,\nonumber\\
\X^n\Lp t\Rp&=\X^n\Lp 0\Rp+\int_0^t h_G\Lp\X^n\Lp s\Rp\Rp ds+\bb^n(t)+\e^n(t) && \mbox{for } t\geq -\pt_n. \label{eq: int eqn old}
\end{align}
\el
\begin{proof}
The linear interpolation functions have simplified definitions in terms of the level crossing times. In particular, from Lemma~\ref{lem: prop tau}~\emph{\eqref{tau up}} and~\emph{\eqref{tau dn}}, we have $\pt_{\ST{n}{t}} \le \pt_n+t < \pt_{\ST{n}{t}+1}$ and on {$[-\pt_n<t<0]$, $\pt_{\STb{n}{-t}} < \pt_n+t \le \pt_{\STb{n}{-t}+1}$.} Using them, we have
\begin{align}
  \X^n(t) &= \begin{cases}
               \X_{\ST{n}{t}} + a_{\ST{n}{t}} \times &\\
               \,\, \Lp h_G\Lp\X_{\ST{n}{t}}\Rp + \bbeta_{\ST{n}{t}} \Rp \frac{(\pt_{n} + t) - \pt_{\ST{n}{t}}}{a_{\ST{n}{t}}}, & \mbox{if } t\geq 0, \\
               \X_{\STb{n}{-t}} + a_{\STb{n}{-t}} \times &\\
               \,\, \Lp h_G\Lp\X_{\STb{n}{-t}}\Rp + \bbeta_{\STb{n}{-t}} \Rp \frac{(\pt_{n} + t) - \pt_{\STb{n}{-t}}}{a_{\STb{n}{-t}}}, & \mbox{if } -\pt_n< t< 0, \\
               \X_0, & \mbox{if } t\leq -\pt_n;
             \end{cases} \label{eq: X}\\
  \Xb^n(t) &= \begin{cases}
               \X_{\ST{n}{t}}, & \mbox{if } t\geq 0, \\
               \X_{\STb{n}{-t}}, & \mbox{if } -\pt_n < t< 0, \\
               0, & \mbox{if } t\leq -\pt_n;
             \end{cases} \label{eq: Xb}\\
  \bb^n(t) &= \begin{cases}
               \sum_{m=n}^{\ST{n}{t}-1} a_m \bbeta_m + a_{\ST{n}{t}} \bbeta_{\ST{n}{t}} \frac{(\pt_{n} + t) - \pt_{\ST{n}{t}}}{a_{\ST{n}{t}}}, & \mbox{if } t\geq 0, \\
               -\Big[ \sum_{m=\STb{n}{-t}+1}^{n-1} a_m \bbeta_m &\\
               \qquad \qquad + a_{\STb{n}{-t}} \bbeta_{\STb{n}{-t}} \frac{\pt_{\STb{n}{-t}+1} - (\pt_{n} + t)}{a_{\STb{n}{-t}}} \Big], & \mbox{if } -\pt_n< t< 0, \\
               -\bb_n, & \mbox{if } t\leq-\pt_n. \nonumber
             \end{cases}
\end{align}

The results follow from $\X_{l} = \X_n + \sum_{m=n}^{l-1} a_m h_G (\X_m) + \sum_{m=n}^{l-1} a_m \bbeta_m$.
\end{proof}

\subsection{Assumptions on delayed sums} \label{subsec: assmp}
To study uniform convergence of $(\bb_n)$ on compacta, we need to consider, for all $T>0$,
$$\sup_{-T\leq t\leq T} \Lone{\bb^n(t)} \leq \sup_{0\leq t\leq T} \Lone{\bb^n(-t)} + \sup_{0\leq t\leq T} \Lone{\bb^n(t)}.$$
Further, using $\STb{n}{-t}=0$ on $[\pt_n+t\le 0]$ from Lemma~\ref{lem: prop tau}~\emph{\eqref{tau dn}}, we have
\begin{equation} \label{eq: B}
 \bb^n(t) = \begin{cases}
               \sum_{m=n}^{\ST{n}{t}-1} a_m \bbeta_m \frac{\pt_{\ST{n}{t}+1}-(\pt_{n} + t)}{a_{\ST{n}{t}}} + &\\
               \qquad \qquad \sum_{m=n}^{\ST{n}{t}} a_m \bbeta_m \frac{(\pt_{n} + t) - \pt_{\ST{n}{t}}}{a_{\ST{n}{t}}}, & \mbox{if } t\geq 0, \\
               -\Bigg[ \sum_{m=\STb{n}{-t}+1}^{n-1} a_m \bbeta_m \frac{(\pt_{n} + t)-\pt_{\STb{n}{-t}}}{a_{\STb{n}{-t}}} + &\\
               \qquad \qquad \sum_{m=\STb{n}{-t}}^{n-1} a_m \bbeta_m \frac{\pt_{\STb{n}{-t}+1} - (\pt_{n} + t)}{a_{\STb{n}{-t}}} \Bigg], & \mbox{if } -\pt_n< t< 0, \\
               -\sum_{m=\STb{n}{-t}}^{n-1} a_m \bbeta_m, & \mbox{if } t\leq-\pt_n.
             \end{cases}
\end{equation}
Hence, we need to assume negligibility of the forward and backward delayed sums, $\sum_{i=n}^{m} a_i\bbeta_i$ for $n\le m\le \ST{n}{T}$ and $\sum_{i=m}^{n-1} a_i\bbeta_i$ for $\STb{n}{T}\le m\le n-1$ respectively. For the deterministic (and hence the almost sure) case considered in \citet[A2.2.4]{KushnerClark} the negligibility assumption was made only on the forward delayed sums, from which the analogous condition on the backward delayed sums can be deduced. However, for convergence in probability, we need to assume negligibility of both the forward and backward delayed sums.
\bass For every $t>0$, the following holds:\label{WeakAssm}
\begin{align*}
\max_{n< m\le\ST{n}{t}}\Lone{\sum_{i=n}^{m-1} a_i\bbeta_i}&\to 0 \quad \text{in probability}
\intertext{and}
\max_{\STb{n}{t}\le m< n-1}\Lone{\sum_{i=m+1}^{n-1} a_i\bbeta_i}&\to 0 \quad \text{in probability}.
\end{align*}
\eass

For applications in the following sections, we shall check the following corollary, which gives conditions equivalent to the assumptions, for appropriate error terms.
\begin{lemma} \label{cor: neg}
Assumption~\ref{WeakAssm}, is equivalent to, for every $t>0$,
\begin{equation} \label{WeakAssm Alt}
\lim_{n\to\infty} \max_{n\leq m\leq \ST{n}{t}} \Lone{\sum_{i=n}^{m} a_i \bbeta_i} = 0 \,\, \text{ and } \,\, \lim_{n\to\infty} \max_{\STb{n}{t}\leq m\leq n-1} \Lone{\sum_{i=m}^{n-1} a_i \bbeta_i} = 0,
\end{equation}
in probability.
\end{lemma}
\begin{proof}
The limits in~\eqref{WeakAssm Alt} trivially imply Assumption~\ref{WeakAssm}.

For converse, simply observe, on the event $[a_{\ST{n}{t}}\leq t]$, which, from Lemma~\ref{lem: prop tau}~\emph{\eqref{scale}}, happens all but finitely often with probability 1, we have
\begin{align*}
\max_{n\leq m\leq \ST{n}{t}} \Lone{\sum_{i=n}^{m} a_i \bbeta_i} &\le \max_{n\le m<\ST{n}{2t}}\Lone{\sum_{i=n}^m a_i\bbeta_i}\\
&= \max_{n< m\le\ST{n}{2t}}\Lone{\sum_{i=n}^{m-1} a_i\bbeta_i}.
\end{align*}
Similarly, on the event $[a_{\STb{n}{t}}\leq t]$, which, from Lemma~\ref{lem: prop tau}~\emph{\eqref{scale}}, also happens all but finitely often with probability 1, we have
\begin{align*}
\max_{\STb{n}{t}\leq m\leq n-1} \Lone{\sum_{i=m}^{n-1} a_i \bbeta_i} &\leq \max_{\STb{n}{2t}< m\le n-1}\Lone{\sum_{i=m}^{n-1} a_i\bbeta_i}\\
&= \max_{\STb{n}{2t}\le m< n-1}\Lone{\sum_{i=m+1}^{n-1} a_i\bbeta_i}.
\end{align*}
The result then follows from Assumption~\ref{WeakAssm}.
\end{proof}

\bark
The negligibility of the backward delayed sums in probability cannot be established from the negligibility of the forward delayed sums in probability, unlike in the case of almost sure or deterministic negligibility, and has to be assumed separately. This is due to the fact that for forward delayed sums the summation starts at a deterministic index and ends at a random one, while the situation is converse for the backward one, leading to the situation where the condition on the forward delayed sum cannot be changed to that for the backward delayed sum.

Assumption~\ref{WeakAssm} is motivated by those in \citet{KushnerClark}. We need to make assumptions on both forward and backward delayed sums unlike those for forward delayed sums alone given in \citet[A4.1.4]{KushnerClark}. In that case, the step sizes and hence the level crossing times are nonrandom. Thus, it is again possible to derive the negligibility in probability of the backward delayed sum in terms of that of the forward delayed sum.
\eark

\subsection{Negligibility of processes}
We are now ready to show the negligibility of the interpolates of the error terms.
\bp Under Assumption~\ref{WeakAssm}, the sequence $\Lp\bb^n\Rp_{n=0}^\infty$ converges to the constant $\0$ process in probability.\label{prop: B}\ep
\begin{proof}
Observe that it is enough to show that, for any $T>0$,
$$\sup_{-T\leq t\leq T} \Lone{\bb^n(t)}\to 0$$
in probability. Next observe that, using~\eqref{eq: B},
\begin{align*}
\sup_{-T\leq t\leq T} \Lone{\bb^n(t)} \leq & \sup_{0\leq t\leq T} \Lone{\bb^n(t)} + \sup_{-T\leq t\leq 0} \Lone{\bb^n(t)}\\
\leq & \max_{n\leq m\leq\ST{n}{T}}\Lone{\sum_{i=n}^m a_i\bbeta_i} + \max_{\STb{n}{T}\leq m\leq n-1}\Lone{\sum_{i=m}^{n-1} a_i\bbeta_i},
\end{align*}
which goes to $0$, using Lemma~\ref{cor: neg}.
\end{proof}

We next show that asymptotically the difference between the processes formed by the integrals $\Lp \int_{0}^{t} h_G(\X^n(u)) du \Rp_{t\in\RR}$ and $\Lp \int_{0}^{t} h_G(\Xb^n(u)) du \Rp_{t\in\RR}$ are negligible.
\bp Under Assumption~\ref{WeakAssm}, the sequence $\Lp\e^n\Rp_{n=0}^\infty$ converges to the constant $\0$ process in probability.\label{prop: e}\ep
\begin{proof}
{As in the proof of Proposition~\ref{prop: B}, we show $\sup_{0\leq t\leq T} \Lone{\e^n(t)}$ and $\sup_{-T\leq t\leq 0} \Lone{\e^n(t)}$ to be negligible, for any $T>0$, in probability.

Note that,
\begin{multline*}
  \sup_{0\leq t\leq T} \Lone{\X^n(t) - \Xb^n(t)}
  \leq \sup_{0\leq t\leq T} \Lone{a_{\ST{n}{t}} h_G(\X_{\ST{n}{t}}) + a_{\ST{n}{t}} \bbeta_{\ST{n}{t}}} \\
  \leq \sup_{\x\in S} h_G(\x) \max_{m\ge n} a_m + \sup_{0\leq t\leq T} \Lone{a_{\ST{n}{t}} \bbeta_{\ST{n}{t}}}.
\end{multline*}
The first term goes to zero almost surely, since, as observed in Remark~\ref{rem: mble}, $\sup_S h_G(\x)$ is a finite random variable, while $a_n\to 0$ almost surely. (Note that $\sup_S h_G(\x)$ is measurable, as noted in Remark~\ref{rem: mble}.) For the second term, observe that
$$a_{\ST{n}{t}} \bbeta_{\ST{n}{t}} = \sum_{i=n}^{\ST{n}{t}} a_i \bbeta_i - \sum_{i=n}^{\ST{n}{t}-1} a_i \bbeta_i$$
giving
$$\sup_{0\leq t\leq T} \Lone{a_{\ST{n}{t}} \bbeta_{\ST{n}{t}}} \leq 2 \max_{n\leq m\leq \ST{n}{T}} \Lone{\sum_{i=n}^{m} a_i \bbeta_i},$$
which converges to $0$ in probability using Assumption~\ref{WeakAssm}.

Next, observe that, the function $h_G$ is continuous for each fixed value of $G$ and hence it is uniformly continuous on the closed bounded set {$S$}. Hence, by negligibility of $\X^n(t) - \Xb^n(t)$ for any $t>0$, we have $h_G(\X^n(t)) - h_G(\Xb^n(t))\to\0$. Further, observe that, the processes $(\X^n(t), t\in\mathbb{R})_{n\ge0}$ and $(\Xb^n(t), t\in\mathbb{R})_{n\ge0}$ take values in the closed bounded set $S$ and hence $\sup_{0\le u\le T}\Lone{\Lp h_G (\X^n(u)) - h_G (\Xb^n(u)) \Rp}$ is bounded. Then, using Dominated Convergence Theorem, we have
\begin{align*}
  \sup_{0\leq t\leq T} \Lone{\e^n(t)} = & \sup_{0\leq t\leq T} \Lone{\int_0^t \Lp h_G (\X^n(u)) - h_G (\Xb^n(u)) \Rp du} \\
  \leq & \int_0^T \Lone{h_G(\X^n(u)) - h_G(\Xb^n(u))} du \to 0.
\end{align*}}

Similarly, on the event $[\pt_n>T]$, which happens all but finitely often with probability $1$, we have
\begin{multline*}
\sup_{0\leq t\leq T} \Lone{\X^n(-t) - \Xb^n(-t)}\\ \leq \sup_{\x\in S} h_G(\x) \max_{m\ge \STb{n}{T}} a_m + 2 \max_{n\leq m\leq \ST{n}{T}} \Lone{\sum_{i=n}^{m} a_i \bbeta_i}.
\end{multline*}
This again is negligible in probability using Assumption~\ref{WeakAssm} and arguing as before, $\sup_{0\leq t\leq T} \Lone{\e^n(-t)}$ is negligible.
\end{proof}

\subsection{Compactness and tightness}
To study the convergence of the random functions {$(\X^n)$}, we use the subsequential limits. Thus, the notion of tightness becomes important. The relevant results for tightness of the corresponding random functions are summarized below. For a discussion of similar results of continuous functions from $[0, \infty)$ to $\RR^K$, see \citet[Chapter~1.3]{StroockVaradhan}. Define the oscillation of a function $\x$ over an interval $[a, b] \subset \RR$ and $\delta>0$ as
$$\osc\Lp\x,[a,b],\delta\Rp:=\sup_{|t-s|\leq\delta, a\leq t<s\leq b}\Lone{\x(t)-\x(s)}.$$
The Arzela-Ascoli-type pre-compactness condition is given by
\begin{prop} \label{prop: cpct}
A sequence $\Lp\x^n\Rp_{n=0}^\infty$ in $C^K$, the space of $\RR^K$ valued continuous functions on $\RR$ endowed with topology of uniform convergence on compacta, is pre-compact iff
\bei
\ii for some (and, hence, for all) $t\in\RR$, $\sup_n \Lone{\x^n(t)}<\infty$ and
\ii for every $T>0$
\[
\lim_{\delta\to 0}\sup_n\osc\Lp\x^n,[-T,T],\delta\Rp=0.
\]
\ee
\end{prop}
The above result easily translates for tightness of $C^K$ valued random elements.
\begin{prop} \label{prop: tight}
Let $\Lp\X^n\Rp_{n=0}^\infty$ be a sequence of random elements in $C^K$. Then the sequence is tight iff
\bei
\ii for some (and, hence, for all) $t\in\RR$, $\Lp\X^n(t)\Rp_{n=0}^\infty$ is tight;
\ii and for every $T>0$ and $\epsilon>0$
\[
\lim_{\delta\to 0}\sup_n\Prob\Lp\osc\Lp\X^n,[-T,T],\delta\Rp>\epsilon\Rp=0.
\]
\ee
\end{prop}

Now it is easy to check the conditions for the sequence of functions $(\X^n)$.
\begin{prop}
{Under Assumption~\ref{WeakAssm}, the sequence $\Lp\X^n\Rp_{0}^\infty$ is tight.}\label{prop: X}
\end{prop}
\begin{proof}
{Note that for any $t$, $\Lp\X^n(t)\Rp_n$ is a sequence of random variables taking values in the closed bounded, {and hence compact,} set $S$ and hence is tight. Next we need to consider the oscillations of $\Lp\X^n\Rp$. It is enough to consider oscillation of $\X^n$ only in the interval $[-\pt_n,\infty)$ because by definition $\X^n(t)=\X_0$ for all $t\leq-\pt_n$. From Lemma~\ref{lem: int eqn}, we know that for $t\geq-\pt_n$
$$\X^n(t) = \X^n\Lp 0\Rp+\int_0^t h_G\Lp\X^n(s)\Rp ds+\bb^n(t)+\e^n(t).$$

For the oscillation of the function in the second term observe that for $t,s\geq-\pt_n$
\begin{align*}
\Lone{\int_0^t h_G\Lp\X^n(u)\Rp du-\int_0^s h_G\Lp\X^n(u)\Rp du} &= \Lone{\int_s^t h_G\Lp\X^n(u)\Rp du}\\
&\leq \sup_{\x\in S} h_G(\x) |t-s|.
\end{align*}
The first term represents functions constant at $\X^n(0)$ which has no oscillation. Hence, we conclude
\begin{multline*}
\osc\Lp\X^n,[-T,T],\delta\Rp\\ \leq \delta \sup_S h_G(\x) + \osc\Lp\bb^n,[-T,T],\delta\Rp + \osc\Lp\e^n,[-T,T],\delta\Rp.
\end{multline*}
Applying Propositions~\ref{prop: B} and~\ref{prop: e} to Proposition~\ref{prop: tight}, the oscillations of $\bb^n$ and $\e^n$ are negligible. Hence the condition in Proposition~\ref{prop: tight}~(ii) follows, cf.\ Remark~\ref{rem: mble}.}
\end{proof}

\subsection{Main results}\label{subsec: result}
We are now ready to state the main theorems for stochastic approximation with random step size and drift.
\bark \label{rem: jt mble}
{Note that, $\Z_G$ will be the limit of $\X^n$ and hence a valid random function. However we have to further assume $\Z_G(0)$ is continuous in $G$.}
\eark

\begin{thm} \label{thm: SA weak}
Let $(\X_n)_{n=0}^\infty$ be a sequence of random variables taking values in a nonrandom closed convex bounded subset $S$ of $\RR^K$. Let $(a_n)_{n=0}^\infty$ be a random sequence of positive numbers satisfying $a_n\to 0$ and $\sum_n a_n = \infty$ almost surely.

Let $\mathfrak G$ be some Polish space and $(G, \x) \mapsto h_G(\x)$ be the drift from $\mathfrak G \times S$ to $\RR^K$, which is jointly continuous in $(G, \x)$.

The error sequence $(\bbeta_n)_{n=0}^\infty$  are also $\RR^K$-valued random variable, which, together with the $\mathfrak G$-valued random variable $G$ are defined on the same probability space as the sequences $(a_n)$ and $(\X_n)$. The sequences $(a_n)$ and $(\bbeta_n)$ satisfy Assumption~\ref{WeakAssm}. They are related through the stochastic approximation equation
$$\X_{n+1} = \X_n + a_n h_G (\X_n) + a_n \bbeta_n.$$

Assume that, for almost all values of $G$, the associated ODE $\dot \X = h_G(\X)$ has a unique solution $\Z_G$, such that $\Z_G(t)$ lies in $S$ for all $(G, t)$. {Also assume that, $\Z_G(0)$ is continuous in $G$.}

Then $\X_n \to \Z_G(0)$ holds in probability, as well as in $L^1$.
\end{thm}
\begin{proof}
First define $\hX^n(t) := \X^n(t) - \Z_G(0)$ and
$$\hH(\x):= h_G\Lp\x + \Z_G(0)\Rp, \quad \text{{ for $(\x, G)$ satisfying $\x+\Z_G(0)\in S$}}.$$
Since, $h_G(\x)$ is jointly continuous in $(G, \x)$ and $\Z_G(0)$ is continuous in $G$, $\hH(\x)$ is also jointly continuous in $(G, \x)$. Further, from~\eqref{eq: int eqn old}, we have
\begin{equation} \label{eq: int eqn}
{\hX^n(t) = \hX^n(0) + \int_0^t \hH\Lp\hX^n(s)\Rp ds + \bb^n(t) + \e^n(t)\mbox{ for } t\geq -\pt_n.}
\end{equation}

Under Assumption~\ref{WeakAssm}, using Propositions~\ref{prop: B},~\ref{prop: e} and~\ref{prop: X}, each component of $\Lp \X^n, \bb^n, \e^n, G \Rp \in C^K \times C^K \times C^K \times \mathfrak G$ is tight, and hence so is $\Lp \X^n, \bb^n, \e^n, G \Rp$ jointly. Thus, for any weakly convergent subsequence of $\Lp \X^{n_k}, \bb^{n_k}, \e^{n_k}, G \Rp$, by Skorohod representation theorem get {$${ \Lp \tX^{k}, \tbb^{k}, \te^{k}, \tG^{k} \Rp} \to \Lp\tX,\tbb,\te,\tG\Rp \quad \text{almost surely},$$ with marginally for each $k$, $${ \Lp \tX^{k}, \tbb^{k}, \te^{n_k}, \tG^{k} \Rp} \stackrel{\mathrm{d}}{=} \Lp \X^{n_k}, \bb^{n_k}, \e^{n_k}, G\Rp,$$ where $\stackrel{\mathrm{d}}{=}$ denotes equality in distribution. {Note that the sequence of random vectors $\Lp \tX^{k}, \tbb^{k}, \te^{k}, \tG^{k} \Rp$ depends on the choice of the subsequence $(n_k)$.} Hence, by {Propositions~\ref{prop: B} and~\ref{prop: e}} we have $\tbb=\0$ and $\te=\0$. Also, defining ${\htX^{k}(t):=\tX^{k}(t)-\Z_{\tG^{k}}(0)}$, we have, by continuity of $\Z_G(0)$ in $G$, ${\htX^{k}} \to \htX := \tX - \Z_{\tG}(0)$.

Define $\ieq{\underline\x}(t) = \int_0^t \hH(\underline\x(u)) du$, where, for each fixed $G$, $\underline\x\in C^K$ actually takes values in $S-\Z_G(0)$. Since $S$, and hence $S-\Z_G(0)$, are closed, bounded for each $G$ and $\hH$ is continuous in $\x$, the values taken by $\hH(\x)$ are also bounded. Hence $\ieq{\underline\x}$ is $C^K$-valued. Also, $\hH$ being jointly continuous, $\ieq{\underline\x}$ is also jointly continuous in $(G, \underline\x)$. So,
\begin{multline*}
\Lp \htX^{k}(t) - \htX^{k}(0) - \b{I}_{\tG^{k}, \htX^{k}}(t) -\tbb^{k}(t) - \te^{k}(t) \Rp_{t\in\RR}\\ \stackrel{\mathrm{d}}{=} \Lp \hX^{n_k}(t) - \hX^{n_k}(0) - \ieq{\hX^{n_k}}(t) - \bb^{n_k}(t) - \e^{n_k}(t) \Rp_{t\in\RR}.
\end{multline*}
{Fix $T>0$. Now the right side vanishes for all $t\in[-T,T]$ on the event $[\pt_{n_k}\ge T]$, whose probability goes to $1$. Thus, for any $T>0$, the right side, and consequently, the left side converges to the constant zero function on the interval $[-T,T]$ in probability. Hence, the left side converges to the constant zero function as a process on the entire real line in probability.}
Since $\ieq{\underline\x}$ is jointly continuous in $(G, \underline\x)$,} taking limit, $\Lp \htX, \0, \0, \tG \Rp$ must also satisfy the integral equation~\eqref{eq: int eqn}, giving $$\dot{\htX} = \widehat{h}_{\tG}\Lp\htX\Rp = h_{\tG}\Lp\htX + \Z_{\tG}(0)\Rp.$$ A change of variable {$\tX = \htX + \Z_{\tG}(0)$} will change the ODE to $\dot{\tX} = h_{\tG}(\tX)$, with $\tX(t)$ in $S$ for all real $t$. Since the solution to this differential equation is unique, we have $\tX \equiv \Z_{\tG}$ almost surely, which gives
{$$\htX^{k}(0) \to \htX(0) = \tX(0) - \Z_{\tG}(0) \equiv \0 \text{ almost surely.}$$ Hence, $\hX^{n_k}(0) \to \0$ in distribution. Since this is true for all subsequential limits, we have $\hX^n(0) = \X_n - \Z_G(0) \to \0$ in distribution. The convergence is actually in probability since the limit is degenerate. Since $\X^n(0)=\X_n$ and $\Z_G(0)$ are both in a bounded set $S$, the limit will also be in $L^1$.}
\end{proof}

\section{Urn Model as Stochastic Approximation} \label{sec: lv}
In this Section, we introduce urn model with random replacement matrix and formally state all the assumptions. We also pose the evolution equation of the urn model as stochastic approximation equation and study the basic properties of the step sizes and the drift. We collect some useful results to be used in this Section and later.

\subsection{Assumptions and evolution equation}
We consider an urn model with finitely many $K$ colors indexed by $\{1, 2, \ldots, K\}$. The composition vector after $n$-th trial will be denoted by $\C_n = (C_{n1}, \ldots, C_{nK})$. We shall denote the total content of the urn after $n$-th trial by $S_n = \sum_{i=1}^K C_{ni}$. The replacement matrix for $n$-th trial will be a possibly random, but non-negative $K\times K$ matrix $\R_n$. For $n$-th trial, $\Chi_n$ will be the $K$-dimensional indicator vector of the color drawn in that trial, which will take value $\e_i$, the $i$-th coordinate vector in $\RR^K$ if color $i$ is drawn in $n$-th trial. Therefore for $n\geq 1$, the urn composition evolves as:
\beq
\C_n = \C_{n-1} + \Chi_n\R_n.
\label{EvolutionEqn}
\eeq
For $n\geq 1$, let $\FF_{n}$ be the sigma-field containing the entire information till time $n$. In particular, {$\FF_n$ will be generated by the collection $\C_0$, $\Lp\R_m\Rp_{m=1}^{n}$ and $\Lp\Chi_m\Rp_{m=1}^{n}$.} Thus the basic assumptions on drawing of colors and reinforcement can be summarized as follows:
\bass \label{Assmp: basic}
The adapted sequence $((\Chi_n, \R_n), \FF_n)$ has distribution, which satisfies:
\bei
\ii {The initial configuration $\C_0$ is nonzero with nonnegative coordinates and each coordinate having finite expectation.}
\ii {For all $n\ge 1$ and all $1\le i, j\le K$, $R_{nij}$ has finite expectation.}
\ii For all $i$, $\Prob(\Chi_n=\b{e}_i|\C_0,(\R_m)_{m=1}^{n-1}, (\Chi_m)_{m=1}^{n-1})=\cslnmi$.
\ii Given $\FF_{n-1}$, $\Chi_n$ and $\R_n$ are conditionally independent.
\ee
\eass
Then $\CExp{\Chi_n}{n-1}=\csl{n-1}$.

While we do not make any distributional assumptions on $\R_n$ other than its conditional independence from $\Chi_n$ given the past, we need the conditional expectations of $\R_n$ to be close to a limiting matrix $\H$ at least in some weak sense. For $n\geq 1$, we call the conditional expectation of the replacement matrices $\H_{n-1} := \CExp{\R_n}{n-1}$ as the \textit{generating matrices}. We assume the generating matrices $\H_n$ to be close to a matrix $\H$. The matrix $\H$ has been taken to be deterministic in the existing literature. Using stochastic approximation with random step size and drift, developed in Section~\ref{sec: sa}, we shall study the case where $\H$ is allowed to be random as well.

We discuss the closeness of $\Lp\H_n\Rp$ to $\H$ in terms of the following norm:

For a matrix $\b{A}$, define
\[
\mynorm{\b{A}} := \max_i\sum_{j}\left|A_{ij}\right|.
\]
Note that $\mynorm{\b{A}}=\sup_{\x\neq\b{0}} \Lone{\x\b{A}}/\Lone{\x}$. Hence $\mynorm{\b{A}}$ is the operator norm of the map $\x\mapsto\x\b{A}$ from {$\ell^1$ to $\ell^1$}. We require the generating matrices to be close to a matrix in this operator norm. We shall also use the notation $\sigma\Lp\b{A}\Rp := \min_i\sum_{j}\left|A_{ij}\right|$.

As is traditional and assumed by \cite{BaiHu} and \cite{LaruellePages}, we shall make the following assumption on $\H$.
\bass \label{Assmp: H}
The matrix $\H$ is possibly random with nonnegative entries and is irreducible almost surely.
\eass

The following properties of $\H$ are easy consequence of Perron-Frobenius theorem.
\begin{lemma} \label{lem: H prop}
The matrix $\H$ will then have the following properties.
\bei
\ii The eigenvalue of $\H$ with largest real part, $\pfval$, is simple, real and positive. All other eigenvalues are less than or equal to $\pfval$ in modulus and have strictly smaller real part. \label{Assmp: eval}
\ii There exist unique left eigenvector $\lpf$ and right eigenvector $\rpf^T$ corresponding to $\pfval$ such that $\lpf$ and $\rpf^T$ have all coordinates strictly positive and are normalized so that $\lpf\b{1}^T=1$ and $\lpf\rpf^T=1$. \label{Assmp: evec}
\ii $0<\sigma(\H) \leq \pfval \leq \rho(\H) < \infty$.
\ee
\end{lemma}

{\cite{Zhang12, Zhang}} assumed that $\H_n$ converges to $\H$ in Cesaro sense in the operator norm, {almost surely,} that is, $\frac{1}{n}\sum_{m=0}^{n-1}\rho\Lp\H_m-\H\Rp \to 0$ almost surely, while \cite{BaiHu} assumed $\sum_n \frac1n \rho \Lp\H_n-\H\Rp < \infty$ {almost surely}, which, by Kronecker's Lemma, implies the assumption made by \citeauthor{Zhang}. \cite{LaruellePages} assumed $\rho\Lp\H_n-\H\Rp\to 0$ {almost surely}, which also implies the assumption of \citeauthor{Zhang}.

\bark
Note that none of the two conditions made by \cite{BaiHu} or \cite{LaruellePages} imply the other. Also the assumption made by \cite{Zhang} is weaker than these two conditions together, i.e., both the conditions of \citeauthor{BaiHu} and \citeauthor{LaruellePages} may fail, yet that of \citeauthor{Zhang} may hold. This can be seen by taking the sequence $\Lp a_n\Rp_{n=0}^\infty$ defined as $a_n=1$ if $n=[k\log k]$ for some integer $k\geq 1$, and $a_n=0$ otherwise. Then $\lim_{n\to\infty}a_n$ does not exist and $\sum_n a_n/n=\infty$. However, if $[k \log k] \leq n-1 < [(k+1)\log (k+1)]$, then $\sum_{m=0}^{n-1} a_m/n \leq k/[(k+1) \log (k+1)] \to 0$.
\eark

To prove convergence in $L^1$ and hence in probability, {we make further assumption as in \cite{Zhang12, Zhang}, where the convergence is now taken in the sense of probability}:
\bass \label{Assmp: weak}
The generating matrix sequence $(\H_n)$ satisfy:
{\bei
\ii[] $\H_n$ converges to $\H$ in Cesaro sense in the operator norm in probability, i.e., 
$$\frac1n \sum_{k=1}^{n} \rho\Lp\H_{k-1}-\H\Rp\to 0 \quad \text{in probability.}$$
\ee}
\eass

{Finally, all the previous analyses of the urn models with random replacement matrices - \cite{BaiHu, LaruellePages, Zhang12, Zhang} - make appropriate assumptions on boundedness of conditional moments of the replacement matrices. For example, \cite{BaiHu} assumes the replacement matrices to be conditionally $L^{2+\epsilon}$-bounded for some $\epsilon>0$, while \cite{LaruellePages,Zhang} need the replacement matrices to be conditionally $L^2$-bounded. \cite{Zhang12} assumed the replacement matrices to be conditionally $L (\log L)^p$-bounded for some $p>1$. Note that, when the replacement matrices are i.i.d.\ independent of the color selection mechanism, the conditional moment bounds hold trivially. We also make an appropriate uniform integrability bound.
\bass \label{Assmp: ui}
The replacement matrix sequence $(\R_n)$ satisfy:
\bei
\ii[] {$\Lp \rho(\R_n) \Rp_n$} is uniformly integrable.
\ee
\eass}

Due to uniform integrability {of Assumption~\ref{Assmp: ui}}, probability convergence in Assumption~\ref{Assmp: weak} can be strengthened to $L^1$ convergence.

\begin{lemma} \label{lem: L1}
Let $\Lp \rho(\R_n) \Rp$ be uniformly integrable and further assume that
$$\frac1n \sum_{k=1}^{n} \rho\Lp\H_{k-1}-\H\Rp\to 0 \quad \text{in probability.}$$
Then, $\rho(\H)$ has first moment finite and
$$\frac1n \sum_{k=1}^{n} \rho\Lp\H_{k-1}-\H\Rp\to 0 \quad \text{in $L^1$.}$$
\end{lemma}
\begin{proof}
Since we have already assumed convergence in probability, it is enough to show that the sequence on left is uniformly integrable. Further, by triangle inequality
\begin{equation}\label{eq: Lp upgrade bd}
\frac1{n} \sum_{k=0}^{n-1} \rho\Lp\H_k-\H\Rp \le \frac1n \sum_{k=0}^{n-1} \rho(\H_k) + \rho(\H).
\end{equation}

We first consider the averages and show them to be uniformly integrable. Note that
\begin{equation} \label{eq: rho compare}
\begin{split}
\rho(\H_n) &= \max_{1\le i\le K} \sum_{j=1}^{K} H_{n;ij} = \max_{1\le i\le K}\Exp \Lp \sum_{j=1}^{K} R_{n+1;ij} \mid \FF_n \Rp\\
&\leq \Exp \Lp \max_{1\le i\le K} \sum_{j=1}^{K} R_{n+1;ij} \mid \FF_n \Rp = \Exp \Lp \rho \Lp \R_{n+1} \Rp \mid \FF_n \Rp.
\end{split}
\end{equation}
Since $\rho(\R_n)$ are uniformly integrable, so are their conditional expectations, giving uniform integrability of $\Lp \rho(\H_n) \Rp$. Since uniform integrability extends to the convex hull of a collection, the average on the right side of~\eqref{eq: Lp upgrade bd}, $\frac1n \sum_{k=0}^{n-1} \rho(\H_k)$ are also uniformly integrable. Then, to complete the proof using the bound in~\eqref{eq: Lp upgrade bd}, it is enough to show $\rho(\H)$ has first moment finite.

Since we have shown $\Lp \rho(\H_n)\Rp$ is uniformly integrable, we have $\Lp \rho(\H_n) \Rp$ is $L^1$ bounded. Hence $\Lp \frac1n \sum_{k=0}^{n-1} \rho(\H_k) \Rp$ is $L^1$ bounded. Further,
$$\left|\frac1n \sum_{k=1}^n \Lp \rho(\H_{k-1}) - \rho(\H) \Rp \right| \le \frac1n \sum_{k=1}^n \rho(\H_{k-1} - \H) \to 0$$
in probability and hence, by Fatou's lemma, $\rho(\H)$ is integrable, as required.
\end{proof}

To apply the above result, we need to adapt the evolution equation \eqref{EvolutionEqn} appropriately. {Let $Y_n$ be the number of balls added in the $n$-th trial. Then $Y_0:=S_0$ and for $n\geq 1$, $Y_n:=S_n-S_{n-1}$.}

We rewrite the evolution equation \eqref{EvolutionEqn} as
\beq
\csd{n} = \csd{n-1} + \frac{1}{S_n}\h\Lp\csd{n-1}\Rp + \frac{\mgdh_n}{S_n} + \frac{\bxi_n}{S_n}
\label{SA}\eeq
where $\h$, $\mgdh_n$, $\bxi_n$ are defined as follows. The drift, indexed by $K\times K$ matrices with nonnegative entries, is defined as
\begin{equation} \label{eq: H def}
\h\Lp\X\Rp := \X\H - \X\Lp\X\H\b{1}^T\Rp.
\end{equation}
For each $n\geq 1$,
\begin{align}
\begin{split} \label{eq: xi def}
\bxi_n := &\csd{n-1}\Lp\H_{n-1}-\H\Rp - \csd{n-1}\Lp\csd{n-1}\Lp\H_{n-1}-\H\Rp\b{1}^T\Rp\\
= &h_{\H_{n-1}-\H} \Lp\csd{n-1}\Rp,
\end{split}
\intertext{and}
\mgdh_n := &\MGD{\Chi_n\R_n-\csd{n-1}Y_n}{n-1}. \label{eq: D def}
\end{align}
This gives the stochastic approximation for the sequence $\Lp\csl{n}\Rp$, which takes values in the closed bounded convex set of probability simplex in $\RR^K$.

\subsection{Drift} \label{subsec: lotka}
{We now study the basic properties of the drift defined in~\eqref{eq: H def} and shall show that the corresponding ODE has unique solution, which satisfy the continuity assumptions of Theorem~\ref{thm: SA weak}. Since $\h(\x)$ is polynomial in $(\H, \x)$, it is jointly continuous.}


Next, we study the ODE associated with the stochastic approximation algorithm. It is a first order quadratic equation of Lotka-Volterra type.
\bp Only solution of $\dot{\X}=\h\Lp\X\Rp$, where $\X(t)$ is a probability vector for all $t$, is $\X(t)=\lpf$ for all $t$.\label{Lotka}
\ep
\begin{proof} Let $\H=\V\J\U$ be a Jordan decomposition of $\H$ such that $\V\U=\U\V=\I$. We may assume $J_{11}=\pfval$, the first row of $\U$ is $\lpf$, first column of $\V$ is $\rpf^T$. Let $\Y(t)=\X(t)\V$. Since $\X$ is a probability vector and $\rpf$ have all coordinates positive, $Y_1(t)=\X(t)\rpf^T$ is positive and bounded away from $0$. Also $\Lp\Y(t):t\in\RR\Rp$ is bounded since $\Lp\X(t):t\in\RR\Rp$ is a probability vector, for all $t$ real. In terms of $\Y$ the ODE becomes $\dot{\Y}=\X\H\V-\X\V\Lp\X\H\b{1}^T\Rp=\Y\J- \Y\Lp\Y\J\U\b{1}^T\Rp$. So $dY_1/dt=Y_1\Lp\pfval- \Y\J\U\b{1}^T\Rp$ and for $i>1$, $dY_i/dt=Y_{i-1}J_{i-1,i}+ Y_i\Lp J_{ii}-\Y\J\U\b{1}^T\Rp$. Hence for $i>1$
\[
\frac{dY_i/Y_1}{dt} = \frac{\dot{Y_i}Y_1-Y_i\dot{Y_1}}{Y_1^2} = \frac{Y_{i-1}}{Y_1}J_{i-1,i} + \frac{Y_i}{Y_1}\Lp J_{ii}- \pfval\Rp.
\]
Let $\Z:=\Y/Y_1$. Then we have $\dot{\Z}=\Z\Lp\J-\pfval\I\Rp$. So
\[
\Z(t)=\Z(0)\exp\Lp t\Lp\J-\pfval\I\Rp\Rp.
\]
Since $\Y$ is bounded and $Y_1$ is bounded away from $0$, $\Z$ is bounded. $Z_1$ is by definition identically $1$. If possible, consider minimum $i>1$ such that $Z_i(0)\neq 0$. Since { $\J-\pfval\I$} is upper triangular and has first row null, same is true for $\exp\Lp t\Lp\J-\pfval\I\Rp\Rp$. So $Z_i(t) = Z_i(0)\exp(t(J_{ii} - \pfval))$. By Proposition~\ref{lem: H prop}~\eqref{Assmp: eval}, $J_{ii}-\pfval$ have negative real part. Hence $Z_i(t)$ is unbounded as $t \to -\infty$, which leads to a contradiction. Hence we have $\Z\equiv\b{e}_1$. So $\X(t)=Y_1(t)\b{e}_1\U=Y_1(t)\lpf$. Since $\X(t)$ is a probability vector for all $t$, we must have $Y_1\equiv 1$, giving $\X\equiv\lpf$.
\end{proof}

We now consider the continuity of the unique solution, which is the constant function $\lpf$ in $\H$, as required in Theorem~\ref{thm: SA weak}. Since we could not locate a ready reference in the literature, we provide a proof below inspired by Proposition~2.14.1 of \cite{resnick}.
\bp
Under Assumption~\ref{Assmp: H}, both $\pfval$ and $\lpf$ are continuous in $\H$. \label{prop: cont}
\ep
\begin{proof}
The continuity of $\pfval$ has been proved by a variety of methods in the literature. A recent simple proof from the first principles is available in \cite{meyer}.

For the proof of continuity of $\lpf$, first observe that
\begin{equation}\label{eq: PF}
\lpf\Lp \pfval \I - \H + \1^T \1 \Rp = \pfval \lpf - \pfval \lpf + \1 = \1,
\end{equation}
since $\lpf \1^T = 1$. We next show that $\Lp \pfval \I - \H + \1^T \1 \Rp$ is invertible. In fact, it is enough to show that the only solution of
\begin{equation} \label{eq: soln}
\Lp \pfval \I - \H + \1^T \1 \Rp \x^T = \0^T
\end{equation}
is $\x^T=\0^T$. Since $\Lp \pfval \I - \H + \1^T \1 \Rp \x^T = \0^T$, we have $$0 = \lpf \Lp \pfval \I - \H + \1^T \1 \Rp \x^T = \1 \x^T$$ or
\begin{equation} \label{eq: zero}
\1^T \1 \x^T = \0^T.
\end{equation}
Then, from~\eqref{eq: soln}, we get $\H \x^T = \pfval \x^T$. Hence, we also have $(2 \pfval)^{-1} (\H + \pfval \I) \x^T = \x^T$. Iterating and averaging we get,
\begin{equation}\label{eq: ave}
\frac1N \sum_{n=0}^N (2\pfval)^{-n} (\H + \pfval \I)^n \x^T = \x^T.
\end{equation}
Now, by Lemma~\ref{lem: H prop}, the matrix $(2 \pfval)^{-1} (\H + \pfval \I)$ has all eigenvalues with nonnegative real parts, has $1$ as a simple eigenvalue with corresponding left and right eigenvectors $\lpf$ and $\rpf^T$ respectively, normalized to $\lpf \1^T = 1 = \lpf \rpf^T$ and have all other eigenvalues with moduli strictly less than $1$. Then by the result on Cesaro summability in \cite[p.~633]{meyerBook}, $\frac1N \sum_{n=0}^N (2\pfval)^{-n} (\H + \pfval \I)^n \to \rpf^T \lpf$. Hence, from~\eqref{eq: ave}, we have
\begin{equation}\label{eq: x}
\rpf^T \lpf \x^T = \x^T.
\end{equation}
Hence, recalling from~\eqref{eq: zero}, $\0^T = \1^T \1 \x^T = \1^T \1 \rpf^T \lpf \x^T$. Since $\rpf^T$ has all entries strictly positive, we also have all entries of $\1^T \1 \rpf^T$ strictly positive. Thus, we must have $\lpf \x^T = 0$. Then,~\eqref{eq: x} gives $\x^T = \rpf^T \lpf \x^T = \0^T$ as required to show $\Lp \pfval \I - \H + \1^T \1 \Rp$ is invertible and we have, from~\eqref{eq: PF},
$$\lpf = \1 \Lp \pfval \I - \H + \1^T \1 \Rp^{-1},$$
which is continuous in $\H$.
\end{proof}

\subsection{Some useful results}
Before considering the step size sequence, we gather some results, which will be useful in studying the step size sequence as well as the error terms.

The first one is about $L^1$ convergence of product of two random elements.

\begin{lemma} \label{L1 prod}
Let $(X_n)$ be a sequence of scalar random variables converging to $X$ in $L^1$, while $(Y_n)$ be a uniformly bounded sequence of random vectors converging to $Y$ in probability. Then $X_n Y_n \to X Y$ in $L^1$.
\end{lemma}
\begin{proof}
We may assume without loss of generality that $(Y_n)$ is uniformly bounded by $1$. Hence same will be true for $Y$. Now, we have $X_n Y_n - XY = (X_n-X)Y_n + X(Y_n-Y)$. The first term is bounded by $|X_n-X|$, which converges to $0$ in $L^1$, while the second term is bounded by $2|X|$, which is integrable. The second term then is negligible in $L^1$ by Dominated Convergence Theorem.
\end{proof}

The following result, that enables us to control the step size sequence and the error terms for the convergence in $L^1$ and hence in probability.

%

{\bp \label{thm: chow_prob}
If $(X_n, \FF_n)_{n=0}^\infty$ is an adapted sequence with $\Lp X_n \Rp$ uniformly integrable, then $\Lp X_n - \CExp{X_n}{n-1} \Rp$ is also uniformly integrable. Further $\frac1n \sum_{i=0}^{n-1} \Lp X_i - \CExp{X_i}{i-1} \Rp \to 0$ in $L^1$.
\ep
The proof of the first result is immediate, while the second one follows from Theorem~2.22 of \cite{HallHeyde}.}

\subsection{Step size}
We consider the stochastic approximation equation~\eqref{SA} with random step size $1/S_{n+1}$. Canonically stochastic approximation equations have non-random step size, the most common step size being $1/(n+1)$. The stochastic approximation equations considered by \cite{LaruellePages}, {\cite{Zhang12}} and \cite{Zhang} used size $1/(n+1)$. The step size $1/S_{n+1}$ has been used to analyze two color urn models with fixed replacement unbalanced matrices in \cite{Renlund, Renlund2}. While the use of the random step sizes simplifies the ODE, we need to carefully establish the properties of the step size sequence $(1/S_{n+1})$, as required in Theorem~\ref{thm: SA weak}.

%


{We proceed to analyze the random step size $1/S_{n+1}$ by comparing it with deterministic step size $1/(n+1)$. We compare the random and deterministic step size sequences through the corresponding level crossing times.} Consider the deterministic level crossing times corresponding to the step size sequence $(1/(n+1))$ and define, for $n\geq 0$ and $t>0$,
\begin{align*}
\st{n}{t}:= &\min \LP m\geq n:\frac{1}{n+1}+\cdots+\frac{1}{m+1}>t\RP\\
= &\max \LP m\geq n: \frac1{n+1} + \cdots + \frac1m \leq t \RP
\intertext{and}
\stb{n}{t}:= &\min\LP 0\leq m\leq n-1: \frac{1}{m+2}+\cdots+\frac{1}{n}\leq t\RP.
\end{align*}
Here we define empty sums to be zero.

We first obtain bounds on these level crossing times.
\bl For the level crossing times for the deterministic step size sequence $(1/(n+1))$, for any $t>0$, we have $n\leq \st{n}{t} \leq n e^{t+1}$, {$\stb{n}{t}\leq n$} for all $n\ge 1$ and $n \leq 2 e^{t+1} \stb{n}{t}$ for all $n\ge e^{t+1}$. \label{lem: step bd}
\el
\begin{proof}
We use the inequality that, for any $m\geq 1$,
\begin{equation}
\log m < 1 + \frac{1}{2} + \cdots + \frac{1}{m} \leq 1 + \log m. \label{eq: log}
\end{equation}
So for any $t>0$,
\[
\log\Lp\st{n}{t}\Rp - (1+\log n) < \frac{1}{n+1} + \cdots + \frac{1}{\st{n}{t}} \leq t,
\]
leading to $n\leq \st{n}{t} \leq n e^{t+1}$ for all $n$. Note that the left inequality follows from the definition and only the right inequality needs to be deduced from the previous one.

For the backward level crossing time, observe that, if $(1/2)+\cdots+(1/n)>t$ holds, then $\stb{n}{t}\ge 1$. Using $(1/2)+\cdots+(1/n) > \log n - 1$, for $\log n - 1 \geq t$ or equivalently $n\geq e^{t+1}$, we have $\stb{n}{t} \ge 1$. In this case, we also have
\[
\log n - (1+\log(\stb{n}{t}+1)) \leq \frac{1}{\stb{n}{t}+2} + \cdots + \frac{1}{n} \leq t,
\]
giving $\stb{n}{t} + 1 \leq n \leq e^{t+1} (\stb{n}{t}+1)$ or $\stb{n}{t} \leq n \leq 2 e^{t+1} \stb{n}{t}$. Again, the left inequality follows from the definition and only the right one needs to be deduced, where we use $\stb{n}{t}\ge 1$.
\end{proof}

We need to show that the deterministic and random step sizes are comparable over a suitable range with high probability. To define the range, we define comparable sequences. {A pair of} nondecreasing, diverging to infinity, sequences $(p_n)$ and $(q_n)$ will be called \textit{comparable} if, there exists $M>1$ such that $p_n\le q_n$ for all $n$ and $q_n\le Mp_n$ eventually for all $n$. From Lemma~\ref{lem: step bd}, it turns out that we can typically take one of the sequences to be the level crossing times of the step size sequence $(1/(n+1))$, while the other can be the usual sequence of integers.

We shall also show that the random and deterministic step size sequences are of same order with high probability on a range given by two comparable sequences. For two comparable sequences $(p_n)$ and $(q_n)$ and two real numbers $A$ and $B$ with $0<B<A$, and for each $n\geq 1$, we define the event $E_n(A,B)$, on which two step size sequences are {of same order}, as
\begin{equation} \label{eq: E def}
E_n(A,B):=\LT B\leq\frac{S_{m+1}}{m+1}\leq A \quad\mbox{for all}\quad p_n\leq m\leq q_n\RT.
\end{equation}
By choosing $A$ and $B$ suitably, we shall show the event to have arbitrarily high probability.
\bp Let $(p_n)$ and $(q_n)$ be {a pair of} comparable sequences. For two real numbers $A$ and $B$ with $0<B<A$, we have
\[
\limsup_{n\to\infty}\Prob(E_n(A,B)^c)\leq \Prob\Lp\sigma(\H)<2B\Rp+\Prob\Lp\rho(\H)>A/2\Rp.
\]
\label{weakstep}
\ep
\begin{proof}
We begin by introducing two notations which will help us study the sequence $S_n$. We define $\widetilde S_n = \sum_{k=1}^{n} \rho (\H_{k-1} - \H)$. We also define the martingale sum as $\Bar S_n = \sum_{k=1}^{n} \Lp Y_k - \Exp Y_k \mid \FF_{k-1} \Rp$. Then we have,
$$S_n - S_0 = \Bar S_n + \sum_{k=1}^{n} \frac{\C_{k-1}}{S_{k-1}} \Lp \H_{k-1} - \H \Rp \b{1}^T + \sum_{k=1}^{n} \frac{\C_{k-1}}{S_{k-1}} \H \b{1}^T.$$
Hence, we get the following bounds:
\begin{equation}
\Bar S_n - \widetilde S_n + n \sigma(\H) \le S_n \le S_0 + \Bar S_n + \widetilde S_n + n \rho(\H). \label{eq: Sn split}
\end{equation}

Next, fix $0<B<A$. Observe that
\beq
\Prob(E_n(A,B)^c)\leq \Prob\Lp\themin\frac{S_{m+1}}{m+1}<B\Rp + \Prob\Lp\themax\frac{S_{m+1}}{m+1}>A\Rp. \label{weakstep1}
\eeq
We handle the two terms on the right side separately and show that their limiting bounds are the corresponding terms in the statement of the Proposition.

Using~\eqref{eq: Sn split}, note that the second term on the right side of~\eqref{weakstep1} satisfies
\begin{multline*}
\Prob\Lp\themax\frac{S_{m+1}}{m+1}>A\Rp\\  \leq \Prob\Lp\frac{S_0}{p_n}>\frac{A}6\Rp + \Prob\Lp\themax\frac{\Bar{S}_{m+1}}{m+1}>\frac{A}{6}\Rp\\  +  \Prob\Lp\themax\frac{\widetilde{S}_{m+1}}{m+1}>\frac{A}{6}\Rp  +  \Prob\Lp\rho(\H)>\frac{A}{2}\Rp.
\end{multline*}
Clearly, the first term on the right goes to zero. Using Doob's maximal inequality and the properties of the sequences $(p_n)$ and $(q_n)$, the second term on the right side gets bounded by, eventually in $n$,
\[
\Prob\Lp\themax\Bar{S}_{m+1}>p_n\frac{A}{6}\Rp \leq \frac{6}{A}\Exp\frac{|\Bar{S}_{q_n+1}|}{p_n} \leq \frac{12M}{A}\Exp\frac{|\Bar{S}_{q_n+1}|}{q_n+1},
\]
which is negligible by Proposition~\ref{thm: chow_prob}. {Note that $0\le Y_n =\Chi_n\R_n\b{1}^T \le \rho(\R_n)$}, which is uniformly integrable by {Assumption~\ref{Assmp: ui}}, and hence its martingale difference sequence is also uniformly integrable. Finally for the third term on the right side, notice that $(\widetilde{S}_n)$ is non-negative and monotonically increasing, and hence, eventually in $n$,
\[
\Prob\Lp\themax\frac{\widetilde{S}_{m+1}}{m+1}>\frac{A}{6}\Rp \leq \Prob\Lp\frac{\widetilde{S}_{q_n+1}}{p_n}>\frac{A}{6}\Rp \leq \Prob\Lp\frac{\widetilde{S}_{q_n+1}}{q_n+1}>\frac{A}{12M}\Rp,
\]
which is negligible again by {Assumption~\ref{Assmp: weak}}. These provide appropriate asymptotic bounds for the second term on the right side of~\eqref{weakstep1}.

We now consider the first term of~\eqref{weakstep1}, which, by~\eqref{eq: Sn split}, satisfies
\begin{multline*}
\Prob\Lp\themin\frac{S_{m+1}}{m+1}<B\Rp\leq\Prob\Lp\themin\frac{\Bar{S}_{m+1}}{m+1}<-\frac{B}2\Rp\\
+\Prob\Lp\themax\frac{\widetilde{S}_{m+1}}{m+1}>\frac{B}2\Rp+ \Prob\Lp\sigma(\H)<2B\Rp.
\end{multline*}
As before the first two terms become negligible by applying Doob's maximal inequality on $(-\Bar S_n)$ and Proposition~\ref{thm: chow_prob} with {Assumption~\ref{Assmp: ui}} for the first term and using {Assumption~\ref{Assmp: weak}} for the second term. These together provide the correct asymptotic bound for the first term on the right side of~\eqref{weakstep1}.
\end{proof}

{Now we are ready to obtain the negligibility of the random step as well as the divergence of the sum of the random step size.}
\bc Under {Assumptions~\ref{Assmp: basic}~-~\ref{Assmp: ui}}, $S_n\to\infty$ and $\sum 1/S_n=\infty$ almost surely. \label{cor: step}\ec
\begin{proof}
Fix $\delta>0$ and $T>0$. Observe, using Lemma~\ref{lem: step bd}, that the sequences $p_n=n$ and $q_n=\st{n}{TA}$ satisfy the conditions of Proposition~\ref{weakstep}. Then, using Proposition~\ref{weakstep}, get $A$, {$B$} and $n$, we have
{\begin{equation} \label{eq: comp}
\Prob\Lp B\le \frac{S_{m+1}}{m+1}\le A \text{ for all $n\leq m<\st{n}{TA}$}\Rp>1-\delta.
\end{equation}}
Note that {on the above event} $$\sum_{m=n}^{\st{n}{TA}}1/S_{m+1}\geq \frac1A \sum_{m=n}^{\st{n}{TA}} 1/(m+1)>\frac{TA}{A}=T,$$ which implies $\sum 1/S_m>T$. So {$\Prob(\sum_m 1/S_m> T)>1-\delta$.} Since $\delta>0$ and $T>0$ are arbitrary, $\sum 1/S_n = \infty$ almost surely. {Also, from~\eqref{eq: comp}, we have $\Prob \Lp S_{n+1} \ge B(n+1) \Rp > 1- \delta$. Hence $S_n \to \infty$ in probability and $S_n$ being monotone increasing, the convergence is also almost sure.}
\end{proof}
Thus, we conclude that the step size sequence $(1/S_{n+1})$ satisfy the conditions of Theorem~\ref{thm: SA weak} under {Assumptions~\ref{Assmp: basic}~-~\ref{Assmp: ui}}.

\section{Analysis of Urn Models using Stochastic Approximation} \label{sec: main}

After checking that the random drift and the random step sizes satisfy the necessary conditions, we are ready to apply Theorem~\ref{thm: SA weak} to urn models with random replacement matrices both under {Assumptions~\ref{Assmp: basic}~-~\ref{Assmp: ui}} and study the behavior of the proportion vector, configuration vector as well as the count vector. We begin by analyzing the error terms in~\eqref{SA} under the {Assumptions~\ref{Assmp: basic}~-~\ref{Assmp: ui}}.

\subsection{Error terms} \label{subsec: prob}
We first prove a general result for negligibility of delayed sums of error terms corresponding to martingale differences.  We then apply the result to specific cases of martingale difference sequences $(\mgdh_n)$ and $(\bxi_n)$.

\subsubsection{Negligibility of delayed sums over a deterministic range}
We first show the negligibility of delayed sums for error terms defined by an arbitrary martingale difference sequence and the random step size sequence $(1/S_{n+1})$. However, the range of summation is given by the level crossing times of the deterministic step size sequence $(1/(n+1))$. In particular, if $(p_n)$ and $(q_n)$ are two comparable sequences and $\Lp(X_n, \FF_n)\Rp$ is a uniformly integrable martingale difference sequence, we consider the delayed sums of the sequence $(X_{n+1}/S_{n+1})$ over the range $p_n$ through $q_n$. We first show the result for the forward delayed sums only. For the first step, we consider the martingale difference sequence to be uniformly bounded and the ratio $(X_{n+1}/S_{n})$, which itself is a martingale difference sequence with respect to the filtration $(\FF_{n+1})$.

\bl Let $\Lp\Lp X_n,\FF_n\Rp\Rp_{n=1}^\infty$ be a uniformly bounded martingale difference sequence. Then, {for any pair of comparable sequences $(p_n)$ and $(q_n)$ and for any $\epsilon>0$,}
\[
\lim_{n\to\infty}\Prob\Lp\themax\left|\sum_{k=p_n}^{m}\frac{X_{k+2}}{S_{k+1}}\right|>\epsilon\Rp=0.
\]
\label{BddMG1}
\el
\begin{proof}
Let $\lambda$ be the uniform bound for the martingale difference sequence. Fix $0<B<A$ and consider the event $E_n(A, B)$ defined in~\eqref{eq: E def}. On this event, we have for $k\geq p_n$, $S_k \geq S_{p_n+1} \ge B (p_n+1)$. Further, it is easy to check that $\Lp \Lp X_{k+2} \chi \Lp S_{p_n+1} \ge B (p_n+1) \Rp / S_{k+1}, \FF_{k+2} \Rp \Rp_{k\ge p_n}$ is a martingale difference sequence. So we have, using Doob's $L^2$ inequality,
\begin{align*}
&\Prob\Lp\themax\left|\sum_{k=p_n}^{m}\frac{X_{k+2}}{S_{k+1}}\right|>\epsilon\Rp\\
\leq &\Prob\Lp E_n(A,B)^c\Rp + \Prob\Lp \themax\left|\sum_{k=p_n}^{m}\frac{X_{k+2}}{S_{k+1}}\right|>\epsilon; E_n(A, B)\Rp\\
\leq &\Prob\Lp E_n(A,B)^c\Rp + \Prob\Lp \themax\left|\sum_{k=p_n}^{m}\frac{X_{k+2}}{S_{k+1}} \chi \Lp S_{p_n+1} \ge B (p_n+1) \Rp \right|>\epsilon\Rp\\
\leq &\Prob\Lp E_n(A,B)^c\Rp + \frac1{\epsilon^2} \sum_{k=p_n}^{q_n} \Exp\Lp \frac{X_{k+2}^2}{S_{k+1}^2} \chi \Lp S_{p_n+1} \ge B (p_n+1) \Rp \Rp\\
\leq &\Prob\Lp E_n(A,B)^c\Rp + \frac{\lambda^2 q_n}{\epsilon^2 B^2 p_n^2} \leq \Prob\Lp E_n(A,B)^c\Rp + \frac{\lambda^2 M}{\epsilon^2 B^2 p_n},
\end{align*}
where the second term goes to $0$ as $n$ goes to infinity. Then the first term is controlled using Proposition~\ref{weakstep} by letting $B\to 0$ and $A\to\infty$.
\end{proof}

We next correct the error term to the required one, but continue to assume the martingale differences to be {uniformly} bounded.
\bl {Let $\Lp\Lp X_n,\FF_n\Rp\Rp_{n=1}^\infty$ be a uniformly bounded martingale difference sequence. Then, for any pair of comparable sequences $(p_n)$ and $(q_n)$ and for any $\epsilon>0$,}
\[
\lim_{n\to\infty}\Prob\Lp\themax\left|\sum_{k=p_n}^{m}\frac{X_{k+1}}{S_{k+1}}\right|>\epsilon\Rp=0.
\]
\label{BddMG2}
\el
\begin{proof}Let $|X_n|\leq\lambda$ for all $n$. Using triangle inequality we get that, for $p_n\leq m\leq q_n$,
\[
{\left|\left|\sum_{k=p_{n}}^{m}\frac{X_{k+1}}{S_{k+1}}\right|-\left|\sum_{k=p_{n}}^{m}\frac{X_{k+1}}{S_{k}}\right|\right| \leq \lambda \sum_{k=p_{n}}^m \Lp \frac{1}{S_{k}}-\frac{1}{S_{k+1}} \Rp \leq \frac{\lambda}{S_{p_n}}.}
\]
Therefore
\[
{\left|\themax\left|\sum_{k=p_{n}}^{m}\frac{X_{k+1}}{S_{k+1}}\right| - \themax\left|\sum_{k=p_{n}}^{m} \frac{X_{k+1}}{S_{k}} \right| \right| \leq \frac{\lambda}{S_{p_n}},}
\]
which is almost surely negligible by {Corollary~\ref{cor: step}. The result then follows by applying Lemma~\ref{BddMG1} for the pair of sequences $(p_n-1)$ and $(q_n-1)$, which are comparable if $(p_n)$ and $(q_n)$ are so.}
\end{proof}

We next replace the condition of boundedness by uniform integrability.
\bp {Let $\Lp\Lp X_n,\FF_n\Rp\Rp_{n=1}^\infty$ be a uniformly integrable sequence of martingale differences. Then, for any pair of comparable sequences $(p_n)$ and $(q_n)$ and for any $\epsilon>0$,}
\[
\lim_{n\to\infty}\Prob\Lp\themax\left|\sum_{k=p_n}^{m}\frac{X_{k+1}}{S_{k+1}}\right|>\epsilon\Rp=0.
\]\label{FwdDet}
\ep
\begin{proof}
Fix an $\eta>0$. Since $\Lp X_n\Rp$ is uniformly integrable, we get $\lambda>0$ such that, $\Exp\Lp|X_n|\chi\Lp|X_n|>\lambda\Rp\Rp<\eta$, for all $n$. Define
\begin{eqnarray*}
U_n:=\MGD{X_n\chi(|X_n|\leq\lambda)}{n-1}\\
V_n:=\MGD{X_n\chi(|X_n|>\lambda)}{n-1}.
\end{eqnarray*}
Note $X_n=U_n+V_n$ as $\CExp{X_n}{n-1}=0$, and $(U_n)$ is a uniformly bounded martingale difference sequence. Then
\begin{multline*}
\Prob\Lp\themax\left|\sum_{k=p_n}^{m}\frac{X_{k+1}}{S_{k+1}}\right|>\epsilon\Rp\\
\leq \Prob\Lp\themax\left|\sum_{k=p_n}^{m}\frac{U_{k+1}}{S_{k+1}}\right|>\frac{\epsilon}{2}\Rp + \Prob\Lp\themax\left|\sum_{k=p_n}^{m}\frac{V_{k+1}}{S_{k+1}}\right|>\frac{\epsilon}{2}\Rp.
\end{multline*}
From Lemma~\ref{BddMG2}, the first term on right becomes negligible.

Fix $A$ and $B$ with {$0<B<A$}. On the event $E_n(A,B)$ we have $1/S_{m+1}\leq 1/(B(m+1))\leq 1/(Bm)$ for $p_n\leq m\leq q_n$. Since $\Exp|V_n|\leq 2\eta$, we then have
\begin{align*}
&\Prob\Lp\themax\left|\sum_{k=p_n}^{m}\frac{V_{k+1}}{S_{k+1}}\right|>\frac{\epsilon}{2}\Rp \leq \Prob(E_n(A,B)^c) + \Prob\Lp \sum_{k=p_n}^{q_n}\frac{|V_{k+1}|}{Bk}>\frac{\epsilon}{2}\Rp\\
&\qquad\leq \Prob(E_n(A,B)^c)+\frac{4\eta}{B\epsilon}\sum_{k=p_n}^{q_n}\frac{1}{k}\\
&\qquad\leq \Prob(E_n(A,B)^c)+\frac{4\eta}{B\epsilon}\Lp 1+\log q_n-\log(p_n-1)\Rp\\
&\qquad\leq \Prob(E_n(A,B)^c)+\frac{4\eta}{B\epsilon}\Lp 2 + \log M \Rp,
\end{align*}
eventually in $n$, since $q_n\leq Mp_n$. We also use~\eqref{eq: log} in the penultimate step. The result then follows by first letting $n\to\infty$, followed by $\eta\to 0$ and then finally, using Proposition~\ref{weakstep}, by letting $B\to 0$ and $A\to\infty$.
\end{proof}

Finally, we get the negligibility for the backward delayed sums as well.
\bp {Let $\Lp\Lp X_n,\FF_n\Rp\Rp_{n=1}^\infty$ be a uniformly integrable sequence of martingale differences. Then, for any pair of comparable sequences $(p_n)$ and $(q_n)$ and for any $\epsilon>0$,}
\[
\lim_{n\to\infty}\Prob\Lp\themax\left|\sum_{k=m+1}^{q_n}\frac{X_{k+1}}{S_{k+1}}\right|>\epsilon\Rp=0.
\]
\label{BwdDet}\ep
\begin{proof}
Using triangle inequality, we have
\[
\left|\sum_{k=m+1}^{q_n}\frac{X_{k+1}}{S_{k+1}}\right| \leq
\left|\sum_{k=p_n}^{m}\frac{X_{k+1}}{S_{k+1}}\right| +
\left|\sum_{k=p_n}^{q_n}\frac{X_{k+1}}{S_{k+1}}\right|,
\]
which gives
\[
\themax\left|\sum_{k=m+1}^{q_n}\frac{X_{k+1}}{S_{k+1}}\right|\leq 2\themax\left|\sum_{k=p_n}^{m}\frac{X_{k+1}}{S_{k+1}}\right|
\]
and the result follows from Proposition~\ref{FwdDet}.
\end{proof}

\subsubsection{Comparing level crossing times}
The above negligibility conditions are applicable to deterministic range only and are applicable to level crossing times for the deterministic step size sequence. To apply the negligibility of forward and backward delayed sums, we need to bound the level crossing times for the random step size sequence by that for deterministic step size sequence. We achieve that on events of arbitrarily high probability.

Fix $0<B<A$ and $T>0$. Observe that, using Lemma~\ref{lem: step bd}, for $n\ge e^{TA+1}$, the sequences $p_n=n$ and $q_n=\st{n}{TA}$ or $p_n=\stb{n}{TA}$ and $q_n=n$ satisfy the conditions of Proposition~\ref{weakstep}. The corresponding events are denoted by
\begin{align*}
\overline{E}_n(A, B, T) &= \LT B\leq \frac{S_{m+1}}{m+1} \leq A \quad \text{for all } n\leq m \leq \st{n}{TA} \RT
\intertext{and}
\underline{E}_n(A, B, T) &= \LT B\leq \frac{S_{m+1}}{m+1} \leq A \quad \text{for all } \stb{n}{TA}\leq m \leq n \RT.
\end{align*}
These events can be made of arbitrarily high probability by choosing {$n$} and $A$ appropriately large and $B$ small enough. On these events, the level crossing times for two sequences of step sizes become comparable.
\bl
Fix $A$ and $B$ with $0<B<A$ and $T>0$. Then, we have, for all large enough $n$, $\ST{n}{T}\leq\st{n}{TA}$ on the event $\overline{E}_n(A, B, T)$, while $\stb{n}{TA} \leq \STb{n}{T}$ on the event $\underline{E}_n(A, B, T)$. \label{lem: compare}
\el
\begin{proof}
On the event $\overline E_n(A,B)$, we have
\[
\sum_{m=n}^{\st{n}{TA}}\frac{1}{S_{m+1}}\geq\sum_{m=n}^{\st{n}{TA}}\frac{1}{(m+1)A}>\frac{TA}{A}=T
\]
and hence $\ST{n}{T}\leq\st{n}{TA}$ holds.

For $n\geq e^{TA+1}$, on the event $\underline E_n(A,B)$, we have
\[
\sum_{m=\stb{n}{TA}}^{n-1}\frac{1}{S_{m+1}}\geq\sum_{m=\stb{n}{TA}}^{n-1}\frac{1}{(m+1)A}>\frac{TA}{A}=T
\]
and hence $\stb{n}{TA}\leq\STb{n}{T}$ holds.
\end{proof}

\subsubsection{Forward Summation of Error Terms}
We combine the results of Proposition~\ref{FwdDet} and Lemma~\ref{lem: compare} and apply them to the sequences $(\mgdh_n)$ and $(\bxi_n)$ to check Assumption~\ref{WeakAssm} -- actually its equivalent conditions given in Lemma~\ref{cor: neg} -- for the forward delayed sums. We first check $(\mgdh_n)$ sequence.
\bp For every $\epsilon>0$ and $T>0$
\[
\lim_{n\to\infty}\Prob\Lp\max_{n\leq m\leq\ST{n}{T}}\Lone{\sum_{k=n}^{m}\frac{\mgdh_{k+1}}{S_{k+1}}}>\epsilon\Rp=0.
\]
\label{fwderrmg}\ep
\begin{proof}
{First using~\eqref{eq: D def}, recall that
\begin{align*}
\mgdh_n &= \MGD{\Chi_n\R_n-\csd{n-1}Y_n}{n-1}\\
&= \Chi_n\R_n - \Chi_n\R_n\b{1}^T - \csd{n-1}\H_{n-1} + \csd{n-1}\csd{n-1}\H_{n-1}\b{1}^T.
\end{align*}
Then, by~\eqref{eq: rho compare}, we have, $$|D_{ni}| \le 2 \rho (\R_n) + 2 \rho (\H_{n-1}) \leq 2 \rho(\R_n) + 2 \Exp \Lp \rho(\R_{n}) \mid \FF_{n-1} \Rp.$$ Hence $(D_{ni})$ is uniformly integrable by Assumption~\ref{Assmp: ui}.}

Fix $T>0$. Now, using triangle inequality we get
\[
\Prob\Lp\max_{n\leq m\leq\ST{n}{T}}\Lone{\sum_{k=n}^{m}\frac{\mgdh_{k+1}}{S_{k+1}}}>\epsilon\Rp
\leq\sum_{i=1}^K\Prob\Lp\max_{n\leq m\leq\ST{n}{T}} \left| \sum_{k=n}^{m} \frac{\mgdhc_{k+1,i}}{S_{k+1}} \right| >\frac{\epsilon}{K}\Rp
\]
Fix an $i$ and let $X_n:=\mgdhc_{ni}$. Then $\Lp\Lp X_n,\FF_n\Rp\Rp_{n=1}^\infty$ is uniformly integrable. On the event $\overline E_n(A,B,T)$ we have $\ST{n}{T}\leq\st{n}{TA}$. { Therefore, we have}
\begin{multline*}
\Prob\Lp\max_{n\leq m\leq\ST{n}{T}}\left|\sum_{k=n}^{m}\frac{X_{k+1}}{S_{k+1}}\right| >\frac{\epsilon}{K}\Rp\\ \leq \Prob(\overline E_n(A,B,T)^c) + \Prob\Lp\max_{n\leq m\leq\st{n}{TA}} \left|\sum_{k=n}^{m}\frac{X_{k+1}}{S_{k+1}}\right| >\frac{\epsilon}{K}\Rp.
\end{multline*}
Letting $n\to\infty$, the second term becomes negligible using Proposition~\ref{FwdDet} {and Lemma~\ref{lem: step bd}}. Further, using Proposition~\ref{weakstep} and letting $B\to 0$ and $A\to\infty$, the first term can also be made negligible.
\end{proof}

Next, we consider the forward delayed sum corresponding to the sequence $(\bxi_n)$.
\bp For every $\epsilon>0$ and $T>0$
\[
\lim_{n\to\infty}\Prob\Lp\max_{n\leq m\leq\ST{n}{T}}\Lone{\sum_{k=n}^{m}\frac{\bxi_{k+1}}{S_{k+1}}}>\epsilon\Rp=0.
\]
\label{fwderrtail}\ep
\begin{proof}
{First, using~\eqref{eq: xi def}, $\Lone{\bxi_{k+1}}\leq 2\rho(\H_k-\H)$ holds. Also recall $\widetilde{S}_n=\sum_{k=1}^n\rho\Lp\H_{k-1}-\H\Rp$. So
\begin{multline*}
\max_{n\leq m\leq\ST{n}{T}}\Lone{\sum_{k=n}^{m}\frac{\bxi_{k+1}}{S_{k+1}}}
\leq  \sum_{k=n}^{\ST{n}{T}}\frac{\Lone{\bxi_{k+1}}}{S_{k+1}}\\
\leq  \frac{1}{S_{n+1}}\sum_{k=n}^{\ST{n}{T}}\Lone{\bxi_{k+1}}
\leq  \frac{2}{S_{n+1}}\sum_{k=n}^{\ST{n}{T}}\mynorm{\H_k-\H}
\leq  2 \frac{\widetilde{S}_{\ST{n}{T}+1}}{S_{n+1}}.
\end{multline*}}
On the event $\overline E_n(A,B,T)$ we have $\ST{n}{T}\leq\st{n}{TA}$. Therefore on the event $\overline E_n(A,B,T)$, we have
\begin{multline*}
     \max_{n\leq m<\ST{n}{T}}\Lone{\sum_{k=n}^{m}\frac{\bxi_{k+1}}{S_{k+1}}}
\leq 2 \frac{\widetilde{S}_{\ST{n}{T}+1}}{S_{n+1}}\\
\leq 2 \frac{\widetilde{S}_{\st{n}{TA}+1}}{(n+1)B}
\leq 2 \frac{\widetilde{S}_{\st{n}{TA}+1}}{\st{n}{TA}+1}\frac{\st{n}{TA}+1}{(n+1)B}.
\end{multline*}
Thus, we have,
\begin{multline*}
\limsup_{n\to\infty}\Prob\Lp\max_{n\leq m<\ST{n}{T}}\Lone{\sum_{k=n}^{m}\frac{\bxi_{k+1}}{S_{k+1}}}>{\epsilon}\Rp\\ \leq \Prob\Lp \overline E_n(A,B,T)^c\Rp + \Prob\Lp \frac{\widetilde{S}_{\st{n}{TA}+1}}{\st{n}{TA}+1} \frac{\st{n}{TA}+1}{(n+1)B} > \frac{\epsilon}2 \Rp.
\end{multline*}
By {Assumption~\ref{Assmp: weak}}, $\widetilde{S}_n/n\to 0$ in probability. Further, observe that, from Lemma~\ref{lem: step bd}, $(\st{n}{TA}+1)/n$ is bounded by $e^{TA+1}+1$. So the second term goes to zero as $n$ goes to infinity. The first term is made negligible using Proposition~\ref{weakstep} and further letting $B\to0$ and $A\to\infty$.
\end{proof}

\subsubsection{Backward Summation of Error Terms}
Assumption~\ref{WeakAssm}, through its equivalent conditions in Lemma~\ref{cor: neg}, for the backward delayed sums can be checked similarly, where we replace the forward level crossing times by the backward ones and Proposition~\ref{FwdDet} by Proposition~\ref{BwdDet}. We simply state the result.

\bp For every $T>0$,
\[
\max_{\STb{n}{T}< m \leq n-1}\Lone{\sum_{k=m}^{n-1}\frac{\mgdh_{k+1}}{S_{k+1}}}\to 0 \quad \text{and} \quad
\max_{\STb{n}{T}<m\leq n-1}\Lone{\sum_{k=m}^{n-1}\frac{\bxi_{k+1}}{S_{k+1}}}\to 0
\]
in probability.\label{bwderrtail}\ep

\subsection{Main Results} \label{subsec: result urn}
Finally, in this Subsection, we collect the results for urn model and prove them. Recall from pg.~\pageref{def N} that $\N_n$ is the color count vector i.e., $\N_n=\sum_{m=1}^n\Chi_m$.

\begin{thm}
Under the {Assumptions~\ref{Assmp: basic}~-~\ref{Assmp: ui}}, the following convergence in $L^1$ and, hence in probability, take place:
\begin{align}
  \frac{\C_n}{S_n} &\to \b{\pi}_{\H}; \label{prop weak}\\
  \frac{S_n}n &\to \lambda_{\H}; \label{total weak}\\
  \frac{\C_n}n &\to \lambda_{\H} \b{\pi}_{\H}; \label{comp weak}\\
  \frac{\N_n}n &\to \b{\pi}_{\H}. \label{count weak}
\end{align}
\end{thm}
\begin{proof}
We have already checked in Proposition~\ref{Lotka} that the ODE $\dot\x = \h(\x)$ associated with stochastic approximation {has the constant function $\lpf$ as the unique probability solution. Also Proposition~\ref{prop: cont} shows $\lpf$ is a continuous function of $\H$.} The conditions on error terms of the stochastic approximation equation~\eqref{SA} has been checked in Subsection~\ref{subsec: prob}. Hence the convergence~\eqref{prop weak} of $\csl{n} \to \lpf$ in probability and in $L^1$ holds using Theorem~\ref{thm: SA weak}.

Using {$0\leq Y_n=\Chi_n\R_n\b{1}^T\leq\mynorm{\R_n}$ almost surely} and {Assumption~\ref{Assmp: ui}}, $(Y_n)$ is a uniformly integrable sequence. Then, using Proposition~\ref{thm: chow_prob}, $S_n/n - (1/n)\sum_{m=1}^n\CExp{Y_m}{m-1}\to 0$ in $L^1$. Therefore, for~\eqref{total weak}, it is enough to show that $(1/n)\sum_{m=1}^n\CExp{Y_m}{m-1}\to\pfval$ in $L^1$. Recall that
\begin{equation} \label{eq: split}
\frac{1}{n}\sum_{m=1}^n\CExp{Y_m}{m-1} = \frac{1}{n}\sum_{m=1}^n\csd{m-1}\Lp\H_{m-1}-\H\Rp\b{1}^T + \frac{1}{n}\sum_{m=1}^n\csd{m-1}\H\b{1}^T.
\end{equation}
By Lemma~\ref{lem: L1}, the first term in~\eqref{eq: split} converges to $0$ in $L^1$. So it is enough to show the second term of~\eqref{eq: split} to have the limit $\lambda_{\H}$ in $L^1$. Then
{\begin{align*}
\left| \frac1n \sum_{m=1}^n \csd{m-1} \H\1^T - \pfval \right| &= \left| \frac1n \sum_{m=1}^n \Lp \csd{m-1} - \lpf \Rp \H\1^T \right|\\
&\le \rho(\H) \frac1n \sum_{m=1}^{n} \Lone{\csd{m-1} - \lpf},
\end{align*}}
which is bounded by integrable $2\rho(\H)$, and hence converges to $0$ in $L^1$, using~\eqref{prop weak} and Dominated Convergence Theorem.

The convergence in $L^1$ in~\eqref{comp weak} will follow from~\eqref{prop weak} and~\eqref{total weak} using Lemma~\ref{L1 prod}.

Finally, for the convergence in~\eqref{count weak}, observe that $\C_n/S_n\to\b{\pi}_{\H}$ in $L^1$ and so does its Cesaro average. So it is enough to show $L^1$ convergence of
$$\frac{\N_n}{n} - \frac{1}{n}\sum_{m=0}^{n-1}\csd{m} = \frac{1}{n}\sum_{m=1}^n\Lp\Chi_m-\csd{m-1}\Rp \to \b{0}.$$
Since $\Lp \Chi_n \Rp$ is bounded and hence uniform integrable, the required $L^1$ negligibility holds by Proposition~\ref{thm: chow_prob}.
\end{proof}


\end{document}